\documentclass[12pt]{amsart}

\usepackage{amsfonts,amsthm,amsmath,amssymb,amscd,mathrsfs}
\usepackage{graphics}
\usepackage{indentfirst}
\usepackage{cite}
\usepackage{latexsym}
\usepackage[dvips]{epsfig}

\newtheorem{theorem}{Theorem}[section]
\newtheorem{remark}{Remark}[section]
\newtheorem{definition}{Definition}[section]
\newtheorem{lemma}[theorem]{Lemma}
\newtheorem{pro}[theorem]{Proposition}

\newcommand{\bt}{\begin{theorem}}
\newcommand{\bl}{\begin{lemma}}
\newcommand{\el}{\end{lemma}}
\newcommand{\et}{\end{theorem}}

\newcommand{\la}{\label}

\newcommand{\bn}{\begin{eqnarray}}
\newcommand{\en}{\end{eqnarray}}
\newcommand{\bnn}{\begin{eqnarray*}}
\newcommand{\enn}{\end{eqnarray*}}

\newcommand{\ba}{\begin{aligned}}
\newcommand{\ea}{\end{aligned}}
\newcommand{\be}{\begin{equation}}
\newcommand{\ee}{\end{equation}}

\renewcommand{\la}{\label}

\newcommand{\Bv}{{\boldsymbol{v}}}

\newcommand{\Bu}{{\boldsymbol{u}}}
\newcommand{\Be}{{\boldsymbol{e}}}
\newcommand{\BF}{{\boldsymbol{F}}}
\newcommand{\BG}{{\boldsymbol{G}}}

\newcommand{\supp}{\text{supp}}
\newcommand{\bBU}{\bar{{\boldsymbol{U}}}}
\newcommand{\BmcF}{{\boldsymbol{\mathcal{F}}}}
\newcommand{\BPsi}{{\boldsymbol{\Psi}}}

\newcommand{\Bo}{{\boldsymbol{\omega}}}

\newcommand{\what}{\widehat}

\begin{document}

\title[Existence and asymptotic behavior of large axisymmetric solutions]
{Existence and asymptotic behavior of large axisymmetric solutions for steady Navier-Stokes system in a pipe}

\author{Yun Wang}
\address{School of Mathematical Sciences, Center for dynamical systems and differential equations, Soochow University, Suzhou, China}
\email{ywang3@suda.edu.cn}

\author{Chunjing Xie}
\address{School of mathematical Sciences, Institute of Natural Sciences,
Ministry of Education Key Laboratory of Scientific and Engineering Computing,
and SHL-MAC, Shanghai Jiao Tong University, 800 Dongchuan Road, Shanghai, China}
\email{cjxie@sjtu.edu.cn}

\begin{abstract}
In this paper, the existence and uniqueness of strong axisymmetric solutions with large flux for the steady Navier-Stokes system in a pipe are established even when the external force is also suitably large in $L^2$. Furthermore, the exponential convergence rate at far fields for the arbitrary steady solutions with finite $H^2$ distance to the Hagen-Poiseuille flows is established as long as the external forces converge exponentially at far fields. The key point to get the existence of these large solutions is the refined estimate for the derivatives in the axial direction of the stream function and the swirl velocity, which exploits the good effect of the convection term. An important observation for the asymptotic behavior of general solutions is that the solutions are actually small at far fields when they have finite $H^2$ distance to the Hagen-Poiseuille flows. This makes the estimate for the linearized problem  play a crucial role in studying the convergence of general solutions at far fields.
\end{abstract}


\thanks{Updated on \today}

\maketitle

\section{Introduction and Main Results}
An important physical problem in fluid mechanics is to study the flows in nozzles. Given an infinitely long nozzle $\tilde{\Omega}$, a natural problem is to investigate the well-posedness theory for the steady Navier-Stokes system
\be \label{NS1}
\left\{
\ba
& ( \Bu \cdot \nabla )\Bu -  \Delta \Bu + \nabla p = 0,\,\,\,\,\text{in}\,\, \tilde{\Omega} \\
& {\rm div}~\Bu = 0,\,\,\,\, \text{in}\,\,\tilde\Omega,
\ea
\right.
\ee
 supplemented with the no slip conditions, i.e.,
 \begin{equation}\label{BC1}
 \Bu=0\,\,\,\,\text{on}\,\,\partial\tilde\Omega,
 \end{equation}
 where $\Bu=(u^x, u^y, u^z)$ denotes the velocity field of the flows.
If $\tilde\Omega$ is a straight cylinder of the form $\Sigma\times \mathbb{R}$ where $\Sigma$ is a smooth two dimensional domain, then there exists a solution $\Bu =(0, 0, u^z(x,y))$ satisfying \eqref{NS1} and \eqref{BC1}. The solution is called the Poiseuille flow and is uniquely determined by the flow flux $\Phi$ defined by
\be\label{fluxBC}
 \Phi= \int_{\Sigma} u^z(x,y) \, dS.
\ee
In particular, if the straight cylinder is a pipe, i.e., $\Sigma$ is the unit disk $B_1(0)$, then the associated Poiseuille flows $\bBU = \bar{U}(r) \Be_z$ have the explicit form as follows,
\be \label{HPflow}
\bar{U}(r) = \frac{2\Phi}{\pi}(1-r^2)  \ \ \ \ \text{with}\,\, r=\sqrt{x^2+y^2},
\ee
which are also called Hagen-Poiseuille flows.

Given an infinitely long nozzle $\tilde\Omega$ tending to a straight cylinder $\Omega$ at far fields, the problem on the well-posedness theory for \eqref{NS1}-\eqref{BC1}    together with the condition
that the velocity field converges to the Posieuille flows in $\Omega$,
  is called Leray's problem nowadays and it was first addressed by Leray (\cite{Leray}) in 1933. The first significant contribution to the solvability of Leray's problem is due to Amick \cite{Amick1, Amick2}, who reduced the proof of existence to the resolution of a well-known variational problem related to the stability of the Poiseuille flow in a straight cylinder. However, Amick left out the investigation of uniqueness and  existence of the solutions with large flux. A rich and detailed analysis of the problem  is due to Ladyzhenskaya and Solonnikov \cite{LS}.  However, the asymptotic far field behavior of the solutions obtained in \cite{LS} is not very clear. Therefore, in order to get a complete resolution for  Leray's problem, the key issue is to study the asymptotic behavior for the solutions of the steady Navier-Stokes equations in infinitely long nozzles.
   The asymptotic behavior for steady Navier-Stokes system in a nozzle was studied extensively in the literature, see \cite{Horgan, HW, LS, AP,Pileckas} and references therein. A classical and straightforward way to prove the asymptotic behavior for steady solutions of Navier-Stokes system is
to derive a differential inequality for the localized energy \cite{Horgan}. This approach was later refined in
 \cite{HW, LS, AP,Pileckas} and the book by Galdi \cite{Galdi}, etc.  The asymptotic behavior obtained via this method is also only for the solutions of steady Navier-Stokes system with small fluxes.  A significant open problem posed in \cite[p. 19]{Galdi} is the global well-posedness for Leray's problem in a general nozzle when the flux $\Phi$ is large.

Another approach to prove the convergence to Poiseuille flows of steady solutions for Navier-Stokes system is the blowup technique.
In fact, with the aid of the compactness obtained in \cite{LS},  global well-posedness for the Leray's problem in a general nozzle tending to a straight cylinder could be established even when the flux $\Phi$ is large, provided that we can prove global uniqueness or some Liouville type theorem for Poiseuille flows in the straight cylinder. In order to study the global uniqueness of Poiseuille flows in a straight cylinder, an important step is to prove the uniqueness in a bounded set in a suitable metric space.  The uniqueness of Hagen-Poiseuille flows in a uniformly small neighborhood (independent of the size of the flux) was obtained in \cite{WXX1}. More precisely, suppose that $\Omega=B_1(0)\times \mathbb{R}$ is a pipe and $\BF=(F^x, F^y, F^z)$ is external force, does the problem
\be \label{NS}
\left\{
\ba
& ( \Bu \cdot \nabla )\Bu -  \Delta \Bu + \nabla p = \BF,\,\,\,\,\text{in}\,\, {\Omega}, \\
& {\rm div}~\Bu = 0,\,\,\,\, \text{in}\,\, \Omega,
\ea
\right.
\ee
supplemented with no-slip boundary condition
\be \label{noslipBC}
\Bu = 0 \ \ \ \ \mbox{on}\ \partial \Omega
\ee
and the flux constraint
\be\label{fluxBC}
 \int_{B_1(0)} u^z(x, y, z) \, dx dy = \Phi
\ee
have a unique solution in the neighborhood of the Hagen-Poiseuille flows when the external force is small? The uniform nonlinear structural stability of Hagen-Poiseuille flows was established in \cite{WXX1} in the axisymmetric setting. It was proved in \cite{WXX1}  that the problem \eqref{NS}-\eqref{fluxBC} has a unique axisymmetric solution $\Bu$ satisfying
  \be \label{thmuniformnonlinear2}
\|\Bu - \bBU \|_{H^{\frac53} (\Omega)} \leq C \|\BF\|_{L^2(\Omega)}
\ee
and
\be \label{thmuniformnonlinear3}
\|\Bu - \bBU \|_{H^2(\Omega)} \leq C (1 + \Phi^{\frac14}) \|\BF\|_{L^2(\Omega)}.
\ee
when the $L^2-$norm of   $\BF$ is smaller than a uniform constant independent of the flux $\Phi$.

The main goal of this paper contains two parts. The first one is to show the existence and uniqueness of strong solutions for the problem \eqref{NS}-\eqref{fluxBC}, when $\BF$ is large in the case that the flux is large.
 The second goal in this paper is to investigate the convergence rate of steady solutions of Navier-Stokes system in a pipe which have finite $H^2$ distance to the Hagen-Poiseuille flows, even when the flows have large fluxes.

Our first main result is stated as follows.
\bt \label{mainresult1}
Assume that $\BF \in L^2(\Omega)$ and $\BF = \BF(r, z)$ is axisymmetric.
There exists a  constant $\Phi_0\geq 1$, such that
if
\be \label{assump1}
\Phi \geq \Phi_0\ \ \ \ \ \mbox{and}\ \ \ \ \|\BF\|_{L^2(\Omega)} \leq  \Phi^{\frac{1}{96}},
\ee
then the problem \eqref{NS}--\eqref{fluxBC} admits a unique axisymmetric solution $\Bu$ satisfying
\be \label{estuH2}
\|\Bu - \bBU \|_{H^{\frac{19}{12}}(\Omega)} \leq C_0 \Phi^{\frac{1}{96}} \ \ \ \ \mbox{and}\ \ \ \ \|u^r\|_{L^2(\Omega)} \leq \Phi^{-\frac{15}{32}},
\ee
where $C_0$ is a constant independent of $\Phi$ and $\BF$. Moreover, the solution $\Bu$ satisfies that
\be \label{estuH3}
\|\Bu - \bBU \|_{H^2(\Omega)} \leq C \Phi^{\frac{7}{24}}.
\ee
\et

We have the following remarks on Theorem \ref{mainresult1}.
\begin{remark}
The crucial point of
 Theorem \ref{mainresult1} is that the external force $\BF$ can be very large  when the flux of the flow is large.
\end{remark}

\begin{remark}
 If $\Phi$ is sufficiently large, $\BF =0$ and  $\|\Bu - \bBU \|_{H^{\frac{19}{12}}(\Omega)} \leq \Phi^{\frac{1}{96}}$, then $\Bu \equiv \bBU$. It means that $\bBU$ is the unique solution in a bounded set with large radius $\Phi^{\frac{1}{96}}$. This can be regarded as a step to get Liouville type theorem for steady Navier-Stokes system in a pipe.
\end{remark}

In case that $\BF$ has additional structure at far fields, we have  the following asymptotic behavior of solutions of Navier-Stokes system in a pipe.
\bt \label{mainresult2} Assume that $\BF \in H^1(\Omega)$ and $\BF = \BF(r, z)$ is axisymmetric. There exists a constant $\alpha_0$ depending only on $\Phi$, such that if $\BF = \BF(r, z)$ satisfies
\be \label{condF1}
\|e^{\alpha |z|} \BF\|_{L^2 (\Omega)} < + \infty,
\ee
with some $\alpha \in (0 , \, \alpha_0)$, and $\Bu$ is an axisymmetric solution to the problem \eqref{NS}--\eqref{fluxBC}, satisfying
\be \label{condu1}
\|\Bu - \bBU\|_{H^2(\Omega)} < + \infty,
\ee
then one has
\begin{equation}
\|e^{\alpha z} (\Bu-\bBU)\|_{H^2(\Omega\cap \{z\geq 0\})}+ \|e^{-\alpha z} (\Bu-\bBU)\|_{H^2(\Omega\cap \{z\leq 0\})} < + \infty.
\end{equation}
\et

There are a few remarks in order.

\begin{remark}
The key point of Theorem \ref{mainresult2} is that there is neither smallness assumption on the flux $\Phi$ nor the smallness on the deviation of $\Bu$ with $\bBU$.
\end{remark}

\begin{remark}
It follows from \eqref{thmuniformnonlinear3} and \eqref{estuH3} that the solutions obtained in \cite{WXX1} and Theorem \ref{mainresult1} satisfy the condition \eqref{condu1}. Hence if $\BF$  in \cite{WXX1} and Theorem \ref{mainresult1} also satisfies  \eqref{condF1}, then the associated solutions must converge to Hagen-Poiseuille flows exponentially fast.
\end{remark}

\begin{remark}
If $\BF$ decays to zero with an algebraic rate, i.e.,  $\BF$ satisfies
\be \label{condF2}
\| z^k \BF\|_{L^2 (\Omega)} < + \infty,
\ee
with some $k\in \mathbb{N}$, then under the condition \eqref{condu1}, using the same idea of the proof for Theorem \ref{mainresult2} yields that the  axisymmetric solution $\Bu$ of the problem \eqref{NS}--\eqref{fluxBC} converges to the Hagen-Poiseuille flows with the same algebraic rates, i.e., $\Bu$  satisfies
\begin{equation}
\| z^k (\Bu-\bBU)\|_{H^2(\Omega)} < + \infty.
\end{equation}
\end{remark}

\begin{remark}
The same asymptotic behavior also holds for the axisymmetric flows in semi-infinite pipes.
\end{remark}

The structure of this paper is as follows. In Section 2, we introduce the stream function formulation for  the linearized problem of axisymmetric Navier-Stokes system and recall the existence results obtained in \cite{WXX1}.  Some good estimates for the derivatives in the axial direction of the stream function and the swirl velocity are established in  Section \ref{someestimates}. These are the key ingredients to get the existence and uniqueness of large solutions when $\BF$ is large. The existence of solutions for the nonlinear problem is obtained via standard iteration  in Section \ref{secnonlinear} . Section \ref{sec-asymp} devotes to the study on the convergence rates of the flows at far fields, where the key observation is that $\Bu-\bBU$ must be small at far fields when the condition \eqref{condu1} holds so that the estimate for the linearized problem can be used. Some important inequalities are collected in Appendix \ref{Appendix}.


\section{Stream function formulation and existence results}\label{Linear}

To get the existence of the Navier-Stokes equations, we start from the following linearized system around Hagen-Poiseuille flows,
\be \label{2-0-1} \left\{ \ba
&\bBU \cdot \nabla \Bv + \Bv \cdot \nabla \bBU - \Delta \Bv + \nabla P = \BF, \ \ \ \mbox{in}\ \Omega, \\
& {\rm div}~\Bv = 0,\ \ \ \mbox{in}\ \Omega,
\ea
\right.
\ee
supplemented
with no-slip boundary conditions and the flux constraint,
\be \label{BC}
\Bv = 0\ \ \ \mbox{on}\ \partial \Omega,\quad \int_{B_1(0)} v^z(\cdot, \cdot, z)\, dS = 0\ \ \ \text{for any}\,\, z \in \mathbb{R}.
\ee

\subsection{Stream function formulation}\label{sec-stream}
In terms of the cylindrical coordinates, an axisymmetric solution $\Bv$ can be written as
\be \nonumber
\Bv = v^r (r, z) \Be_r + v^z (r, z) \Be_z +v^\theta (r, z) \Be_\theta.
\ee
Then the linearized equation for the Navier-Stokes system \eqref{NS} around  Hagen-Poiseuille flows can be written as
\be \label{2-0-1-1}
\left\{
\ba
& \bar U(r)  \frac{\partial v^r}{\partial z} + \frac{\partial P}{\partial r} -\left[ \frac{1}{r} \frac{\partial}{\partial r}\left(
r \frac{\partial v^r}{\partial r} \right) + \frac{\partial^2 v^r}{\partial z^2} - \frac{v^r}{r^2} \right] = F^r  \ \ \ \mbox{in}\ D ,\\
& v^r \frac{\partial \bar U }{\partial r} + \bar U(r)  \frac{\partial v^z}{\partial z} + \frac{\partial P}{\partial z}
- \left[ \frac{1}{r} \frac{\partial }{\partial r} \left( r \frac{\partial v^z}{\partial r}\right) + \frac{\partial^2 v^z}{\partial z^2} \right] = F^z \ \ \mbox{in}\ D ,                     \\
& \partial_r v^r + \partial_z v^z + \frac{v^r}{r} =0\ \ \ \mbox{in}\ D
\ea \right.
\ee
and
\be \label{vswirl}
\bar U(r)  \partial_z  v^\theta - \left[ \frac{1}{r} \frac{\partial }{\partial r} \left( r \frac{\partial v^\theta}{\partial r}\right) + \frac{\partial^2 v^\theta}{\partial z^2} - \frac{v^\theta}{r^2} \right] =  F^\theta \ \ \ \mbox{in}\ \ D.
\ee
Here $F^r$, $F^z$, and $F^\theta$ are the radial, axial, and azimuthal component of $\BF$, respectively, and $D=\{(r, z): r\in (0, 1), z\in \mathbb{R}\}$.
The no-slip boundary conditions and the flux constraint \eqref{BC}  become
\be \label{BC-1}
v^r(1, z)  = v^z(1, z) = 0,\quad  \int_0^1 r v^z(r, z)\, dr = 0,
\ee
and
\begin{equation}\label{BC-swirl}
v^\theta (1, z) = 0.
\end{equation}

It follows from  the third equation in \eqref{2-0-1-1} that there exists a stream function $\psi(r, z)$ satisfying
\be \label{2-0-4}
v^r =  \partial_z \psi \ \ \text{and} \ \ v^z = - \frac{\partial_r (r \psi) }{r}.
\ee
The azimuthal vorticities of $\Bv$ and $\BF$ are defined as
\be \nonumber
\omega^\theta= \partial_z v^r - \partial_r v^z= \frac{\partial }{\partial r}  \left(  \frac1r \frac{\partial }{\partial r} (r \psi) \right) + \partial_z^2 \psi \ \ \ \ \mbox{and}\ \ \ \ f= \partial_z F^r - \partial_r F^z,
\ee
respectively.
It follows from the first two equations in \eqref{2-0-1-1} that
\be \label{2-0-2}
\bar U(r)  \partial_z \omega^\theta - \left(\partial_r^2 + \partial_z^2 + \frac{1}{r}  \partial_r \right)\omega^\theta
+ \frac{\omega^\theta}{r^2}  = f.
\ee
Denote
\be \nonumber
\mathcal{L} = \frac{\partial}{\partial r} \left(  \frac1r \frac{\partial}{\partial r}(r \cdot)      \right) = \frac{\partial^2}{\partial r^2} + \frac1r \frac{\partial}{\partial r} - \frac{1}{r^2}.
\ee
Hence $\omega^\theta = (\mathcal{L} + \partial_z^2) \psi$ and  $\psi$ satisfies the following fourth order equation,
\be \label{2-0-4-1}
\bar U(r)  \partial_z( \mathcal{L} + \partial_z^2) \psi -
( \mathcal{L} + \partial_z^2)^2 \psi = f.
\ee

 Next, we derive the boundary conditions for $\psi$. As discussed in \cite{Liu-Wang}, in order to get classical solutions, some compatibility conditions at the axis should be imposed. Assume that $\Bv$ and the vorticity $\Bo$ are continuous so that $v^r(0, z)$ and $\omega^\theta(0, z)$ should vanish. This implies
\be \nonumber
\partial_z \psi(0, z) = (\mathcal{L} + \partial_z^2)\psi (0, z) = 0.
\ee
Without loss of generality, one can  assume that $\psi(0, z) = 0$. Hence, the following compatibility condition holds at the axis,
\be \label{2-0-4-2}
\psi(0, z) = \mathcal{L} \psi(0, z) = 0.
\ee
On the other hand, it follows from \eqref{BC-1} that
\be \nonumber
\int_0^1 \partial_r (r \psi ) (r, z)\, dr =-  \int_0^1 r v^z \, dr = 0.
\ee
This, together with \eqref{2-0-4-2}, gives
\be \label{2-0-4-3}
\psi(1, z) = \lim_{r\rightarrow 0+ } ( r \psi)  (r, z)  = 0.
\ee
Moreover, according to the homogeneous boundary conditions for $\Bv$, one has
\be \nonumber
\frac{\partial}{\partial r}(r \psi)  |_{r=1} = r v^z  |_{r=1} = 0.
\ee
This, together with \eqref{2-0-4-3}, implies
\be \label{2-0-4-4}
\frac{\partial \psi}{\partial r}(1, z) = 0.
\ee

Meanwhile, if $\Bv$ is continuous, then the compatibility conditions for $\Bv$ obtained in \cite{Liu-Wang} implies  $v^\theta (0, z) = 0$. Hence  $v^\theta$ satisfies the following problem
\be \label{swirlsystem}
\left\{   \ba
& \bar U(r)  \partial_z  v^\theta - \left[ \frac{1}{r} \frac{\partial }{\partial r} \left( r \frac{\partial v^\theta}{\partial r}\right) + \frac{\partial^2 v^\theta}{\partial z^2} - \frac{v^\theta}{r^2} \right] =  F^\theta \ \ \ \mbox{in}\ \ D, \\
& v^\theta(1, z)= v^\theta(0, z) = 0.
\ea
\right.
\ee

Now let us introduce some notations and recall the existence results for $\psi$ and $\Bv^\theta$ obtained in \cite{WXX1}. For a given function $g(r, z)$, define its Fourier transform with respect to $z$ variable by
\be  \nonumber
\hat{g}(r, \xi) = \int_{\mathbb{R}^1} g(r, z) e^{-i \xi z} dz.
\ee
 $\Re g$ and $\Im g$  denote the real and imaginary part of a function or a number $g$, respectively.

\begin{definition}\label{def1}
Define a function space $C_*^\infty(D)$ as follows
\be \nonumber
C_*^\infty(D )  = \left\{ \varphi(r, z) \left|  \ba & \ \varphi \in C_c^\infty([0, 1] \times \mathbb{R} ), \
\varphi(1, z) = \frac{\partial \varphi}{\partial r}(1, z) =0, \\
&\ \  \   \mbox{and}\ \lim_{r\rightarrow 0+ } \mathcal{L}^k \varphi(r, z)
= \lim_{r \rightarrow 0+ } \frac{\partial}{\partial r} ( r \mathcal{L}^k \varphi ) (r, z) = 0, \ k \in \mathbb{N}
\ea \right.
\right\}.
 \ee
The $H_r^3(D)$-norm and  $H_r^4(D)$-norm are defined as follows,
\be \label{5-100} \ba
\| \varphi\|_{H_r^3(D)} : = & \int_{-\infty}^{+\infty} \int_0^1 \left[ \left| \frac{\partial}{\partial r }( r \mathcal{L} \hat{\varphi})  \right|^2 \frac1r  + \xi^2 |\mathcal{L} \hat{\varphi} |^2 r + \xi^4 \left| \frac{\partial}{\partial r} (r \hat{\varphi} ) \right|^2 \frac1r + \xi^6 |\hat{\varphi} |^2 r  \right] \, dr \\
&\ \ + \int_{-\infty}^{+\infty} \int_0^1  \left[  |\mathcal{L} \hat{\varphi} |^2 r  + \xi^2 \left| \frac{\partial}{\partial r} ( r \hat{\varphi})  \right|^2 \frac{1}{r}  + \xi^4 |\hat{\varphi}|^2 r \right] \, dr \\
&\ \ + \int_{-\infty}^{+\infty} \int_0^1 \left[ \left|  \frac{\partial}{\partial r} (r \hat{\varphi}) \right|^2 \frac1r + \xi^2 |\hat{\varphi}|^2 r + |\hat{\varphi}|^2 r      \right] \, dr .
\ea
\ee
$H_*^3(D)$  denotes the closure of $C_*^\infty(D)$ under the $H_r^3(D)$-norm.
Furthermore, $L_r^2(D)$ is the completion of $C^\infty(D)$ under the $L_r^2(D)$-norm defined as follows
\be \nonumber
 \ \|f\|_{L_r^2(D )}^2 = \int_{-\infty}^{+\infty} \int_0^1  |f|^2  r \, dr dz .
\ee
\end{definition}

The existence of solutions for the problems \eqref{2-0-4-1}--\eqref{2-0-4-4} and \eqref{swirlsystem} has been established in \cite{WXX1}.
\begin{pro} \label{existence-stream} \cite[Theorem 1.1]{WXX1} Assume that  $\BF^* = F^r \Be_r + F^z \Be_z  \in L^2 (\Omega)$ and $\BF^*$ is axisymmetric. There exists a unique solution $\psi \in H_*^3(D) $ to the linear system \eqref{2-0-4-1}--\eqref{2-0-4-4},  and a positive constant $C_1$ independent of $\BF^* $ and $\Phi$, such that
\begin{equation}\nonumber
\| \Bv^* \|_{H^{\frac53}(\Omega)} \leq C_1 \|\BF^*\|_{L^2(\Omega)},
\end{equation}
and
\be \nonumber
\|\Bv^* \|_{H^2(\Omega)} \leq C_1 (1 + \Phi^{\frac14} ) \|\BF^* \|_{L^2(\Omega)},
\ee
where $\Bv^* = v^r \Be_r + v^z \Be_z = \partial_z \psi \Be_r - \frac{\partial_r (r\psi)}{r} \Be_z . $
\end{pro}

\begin{pro}\label{existence-swirl}\cite[Proposition 4.6]{WXX1}
Assume that $\BF^\theta = F^\theta \Be_\theta \in L^2(\Omega)$ and $\BF^\theta = \BF^\theta (r, z)$ is axisymmetric. There exist a unique solution $v^\theta $ to the linear problem \eqref{swirlsystem} and a positive constant $C_2$ independent of $F^\theta$ and $\Phi$, such that
\be \nonumber
\|\Bv^\theta \|_{H^2(\Omega)} \leq C_2 \|\BF^\theta \|_{L^2(\Omega)},
\ee
where $\Bv^\theta = v^\theta \Be_\theta.$
\end{pro}


\section{Some Refined estimates for solutions of Linearized problem}\label{someestimates}

Propositions \ref{existence-stream}-\ref{existence-swirl} provide some uniform estimates for $\psi$ and $\Bv^\theta$. They are the key ingredients to get the existence and uniqueness of solutions of the steady Navier-Stokes system, when $\BF$ is uniformly small\cite{WXX1}. In this section, we give some refined estimates, especially for the $z$-derivatives of $\psi$ and $\Bv^\theta$, which yield the existence of large solutions of steady Navier-Stokes system even when the external force is large. These estimates also show the stabilizing effect of the linearized convection term when $\Phi$ is large.

We take the Fourier transform with respect to $z$ for the equation \eqref{2-0-4-1} . For every fixed $\xi$, $\hat{\psi}$ satisfies
\be \label{2-0-8}
i\xi \bar{U}(r) ( \mathcal{L} - \xi^2) \hat{\psi}  - ( \mathcal{L} - \xi^2 )^2 \hat{\psi} = \hat{f}= i \xi \widehat{F^r} - \frac{d}{dr} \widehat{F^z}.
\ee
Furthermore, the boundary conditions \eqref{2-0-4-2}-\eqref{2-0-4-4} can be written as
\be \label{FBC}
 \hat{\psi}(0) = \hat{\psi} (1) = \hat{\psi}^{\prime} (1) =  \mathcal{L} \hat{\psi}(0) = 0.
\ee


First, let us recall the a priori estimates obtained in \cite{WXX1} which hold for every $\xi \in \mathbb{R}$.
\begin{pro}\label{propcase0}\cite[Section 6]{WXX1}
Let $\hat{\psi}(r, \xi)$ be a smooth solution of the problem \eqref{2-0-8}--\eqref{FBC}. Then it holds
\be
\int_0^1 |\mathcal{L} \hat{\psi}|^2 r \, dr + \xi^2 \int_0^1 \left| \frac{d}{dr} (r \hat{\psi} ) \right|^2 \frac1r \, dr + \xi^4 \int_0^1 |\hat{\psi}|^2 r \, dr
\leq C \int_0^1 ( |F^r|^2 + |F^z|^2 ) r \, dr.
\ee
\end{pro}

The next two propositions give some further estimates for $\psi$, especially the $z$-derivatives of $\psi$.
\begin{pro} \label{further-estimate-1}
Assume that $\BF^r = F^r \Be_r \in L^2(\Omega)$, the solution $\psi$ of the problem
\be \label{4-11}
\left\{  \ba & \bar{U}(r) \partial_z (\mathcal{L} + \partial_z^2) \psi - (\mathcal{L} + \partial_z^2 )^2 \psi = \partial_z F^r, \\
& \psi(0) = \psi(1) = \frac{\partial}{\partial r } \psi(1) = \mathcal{L} \psi (0) = 0,
\ea
\right.
\ee
satisfies
\be \nonumber
\|\Bv^* \|_{L^2(\Omega)} \leq C \Phi^{-\frac23} \|\BF^r\|_{L^2(\Omega)} \ \ \ \ \mbox{and}\ \ \ \ \|\Bv^*\|_{H^{\frac{19}{12}}(\Omega)} \leq C \Phi^{-\frac{1}{30}} \|\BF^r\|_{L^2(\Omega)}.
\ee
\end{pro}

\begin{proof}
Taking the Fourier transform with respect to $z$ for the system \eqref{4-11} yields that for every fixed $\xi$, $\hat{\psi}(r, \xi)$ satisfies
\be \label{4-12}
i \xi \bar{U}(r) (\mathcal{L}- \xi^2) \hat{\psi} - ( \mathcal{L} - \xi^2)^2 \hat{\psi}= i \xi \widehat{F^r}.
\ee
Multiplying \eqref{4-12} by $r\overline{\hat{\psi}}$ and integrating the resulting equation over $[0, 1]$ give
\be \nonumber
\int_0^1 i \xi \bar{U} (r) (\mathcal{L}- \xi^2) \hat{\psi } \overline{\hat{\psi}} r \, dr -
\int_0^1 (\mathcal{L} - \xi^2 )^2 \hat{\psi} \overline{\hat{\psi}} r \, dr = i \xi \int_0^1 \widehat{F^r} \overline{\hat{\psi}} r \, dr.
\ee
It follows from the direct computations and integration by parts that one has
\be \label{4-14}
\xi \int_0^1 \frac{\bar{U}(r) }{r } \left|  \frac{d}{dr} (r \hat{\psi} ) \right|^2 \, dr + \xi^3 \int_0^1 \bar{U}(r) |\hat{\psi}|^2 r \, dr = -\Re \int_0^1 \xi \widehat{F^r} \overline{\hat{\psi}} r\, dr.
\ee
This, together with  together with Lemma \ref{lemmaHLP}, gives
\be \nonumber \ba
 \Phi \int_0^1 \left| \frac{d}{dr} ( r \hat{\psi})   \right|^2 \frac{1- r^2}{r} \, dr & \leq C \int_0^1 |\what{F^r} | |\hat{\psi}| r \, dr \\
&  \leq C \left( \int_0^1 |\what{F^r} |^2 r \, dr  \right)^{\frac12} \left( \int_0^1  \left| \frac{d}{dr} ( r \hat{\psi})   \right|^2 \frac{1- r^2}{r} \, dr  \right)^{\frac12}.
 \ea
\ee
Hence, we have
\be \label{4-16}
\int_0^1  \left| \frac{d}{dr} ( r \hat{\psi})   \right|^2 \frac{1- r^2}{r} \, dr \leq C \Phi^{-2} \int_0^1 |\what{F^r}|^2 r \, dr.
\ee
It follows from  \ref{weightinequality}, Lemma \ref{lemma1}, and Proposition \ref{propcase0} that
\be \label{4-17} \ba
\int_0^1 \left| \frac{d}{dr} (r \hat{\psi}) \right|^2 \frac1r \, dr & \leq C \left(  \int_0^1  \left| \frac{d}{dr} ( r \hat{\psi})   \right|^2 \frac{1- r^2}{r} \, dr  \right)^{\frac23} \left( \int_0^1 |\mathcal{L} \hat{\psi} |^2 r \, dr  \right)^{\frac13} \\
& \leq C \Phi^{-\frac43} \int_0^1 |\what{F^r}|^2 r \, dr.
\ea
\ee
This implies
\be \label{4-18}
\|v^z \|_{L^2(\Omega)}^2 = \int_{-\infty}^{+\infty} \int_0^1 \left| \frac{d}{dr} (r \hat{\psi} )  \right|^2 \frac1r \, dr d\xi
\leq C  \Phi^{-\frac43} \|\BF^r \|_{L^2(\Omega)}^2.
\ee

Similarly, the equality \eqref{4-14}, together with Lemma \ref{lemmaHLP}, gives
\be \label{4-19} \ba
\Phi \xi^2 \int_0^1 (1 - r^2) |\hat{\psi}|^2 r \, dr
& \leq C \left( \int_0^1 |\what{F^r} |^2  r \, dr  \right)^{\frac12} \left( \int_0^1 \left| \frac{d}{dr} (r \hat{\psi} )  \right|^2 \frac{1-r^2}{r} \, dr  \right)^{\frac12} \\
& \leq C \Phi^{-2} \int_0^1 |\what{F^r} |^2 r \, dr.
\ea
\ee
It follows from Lemma \ref{weightinequality}, Lemma \ref{lemma1}, and Proposition \ref{propcase0} again that one has
\be \label{4-20} \ba
\xi^2 \int_0^1 |\hat{\psi}|^2 r \, dr &
\leq C \left( \xi^2 \int_0^1 (1  - r^2) |\hat{\psi}|^2 r \, dr   \right)^{\frac23} \left( \xi^2 \int_0^1 \left| \frac{d}{dr} (r \hat{\psi} ) \right|^2 \frac1r \, dr  \right)^{\frac13} \\
& \leq C \Phi^{-\frac43} \int_0^1 |\what{F^r}|^2 r \, dr .
\ea
\ee
This implies
\be \label{4-21}
\|v^r \|_{L^2(\Omega)}^2 = \int_{-\infty}^{+\infty} \int_0^1 \xi^2 |\hat{\psi}|^2 r \, dr d\xi \leq C \Phi^{-\frac43} \|\BF^r\|_{L^2(\Omega)}^2.
\ee

By the interpolation between $L^2(\Omega)$ and $H^{\frac53}(\Omega)$, one has
\be \label{4-22}
\|\Bv^*\|_{H^{\frac{19}{12}} (\Omega)}
\leq C \|\Bv^*\|_{L^2(\Omega)}^{\frac{1}{20}} \|\Bv^*\|_{H^{\frac53}(\Omega)}^{\frac{19}{20}} \leq C \Phi^{-\frac{1}{30}} \|\BF^r\|_{L^2(\Omega)}.
\ee
This finishes the proof of Proposition \ref{further-estimate-1}.
\end{proof}

Next, we study the case when $\BF$ has only axial component.
\begin{pro} \label{further-estimate-2}
Assume that $\BF^z = F^z \Be_z \in L^2(\Omega)$, the solution $\psi$ of the following problem
\be \label{4-31}
\left\{  \ba & \bar{U}(r) \partial_z (\mathcal{L} + \partial_z^2) \psi - (\mathcal{L} + \partial_z^2 )^2 \psi = - \partial_r  F^z , \\
& \psi(0) = \psi(1) = \frac{\partial}{\partial r } \psi(1) = \mathcal{L} \psi (0) = 0,
\ea
\right.
\ee
satisfies
\be \nonumber
\|\Bv^r \|_{L^2(\Omega)} \leq C \Phi^{-\frac12} \|\BF^z\|_{L^2(\Omega)} \ \ \ \ \mbox{and}\ \ \ \ \|\Bv^r \|_{H^{\frac{19}{12}}(\Omega)} \leq C \Phi^{-\frac{1}{40}} \|\BF^z\|_{L^2(\Omega)},
\ee
where $\Bv^r = v^r \Be_r = \partial_z \psi \Be_r$.
\end{pro}
\begin{proof}
Taking the Fourier transform with respect to $z$ for the system \eqref{4-31} yields that for fixed $\xi $, $\hat{\psi}(r, \xi)$ satisfies
\be \label{4-32}
i \xi \bar{U}(r) (\mathcal{L}- \xi^2) \hat{\psi} - ( \mathcal{L} - \xi^2) \hat{\psi}= -\frac{d}{dr}  \widehat{F^z}.
\ee
Multiplying \eqref{4-32} by $r\overline{\hat{\psi}}$ and integrating the resulting equation over $[0, 1]$ give
\be \label{4-33}
\xi \int_0^1 \frac{\bar{U}(r) }{r } \left|  \frac{d}{dr} (r \hat{\psi} ) \right|^2 \, dr + \xi^3 \int_0^1 \bar{U}(r) |\hat{\psi}|^2 r \, dr = \Im \int_0^1  \widehat{F^z} \frac{d}{dr} ( r \overline{\hat{\psi}} ) \, dr.
\ee
This, together with  Proposition \ref{propcase0}, implies
\be \label{4-34}
\Phi \xi^2 \int_0^1 \left| \frac{d}{dr}(r \hat{\psi} ) \right|^2 \frac{1- r^2}{r} \, dr
\leq C |\xi| \int_0^1 |\widehat{F^z}| \left| \frac{d}{dr}(r \hat{\psi}) \right| \, dr \leq C \int_0^1 |\what{F^z}|^2 r \, dr.
\ee
Applying Lemma \ref{lemmaHLP} yields
\be \label{4-35}
\xi^2 \int_0^1 |\hat{\psi}|^2 r \, dr \leq C \xi^2 \int_0^1 \left| \frac{d}{dr} (r \hat{\psi} )  \right|^2 \frac{1-r^2}{r} \, dr
\leq C \Phi^{-1} \int_0^1 |\widehat{F^z}|^2 r \, dr .
\ee
Therefore, one has
\be \label{4-36}
\| \Bv^r\|_{L^2(\Omega)} \leq C \Phi^{-\frac12} \|\BF^z\|_{L^2(\Omega)}.
\ee

It follows from the interpolation and Proposition \ref{existence-stream} that one has
\be \label{4-37}
\|\Bv^r\|_{H^{\frac{19}{12}}(\Omega)} \leq C \Phi^{-\frac{1}{40}} \|\BF^z \|_{L^2(\Omega)}.
\ee
This finishes the proof of Proposition \ref{further-estimate-2}.
\end{proof}

Now we are in position to analyze $v^\theta$. Taking the Fourier transform with respect to $z$ for the problem \eqref{swirlsystem} gives
\be \label{swirl-Fourier}
\left\{
\ba & i \xi \bar{U}(r) \what{v^\theta} - ( \mathcal{L} - \xi^2) \what{v^\theta} = \what{F^\theta}, \\
& \what{v^\theta}(1 ) = \what{v^\theta}(0) = 0 .
\ea
\right.
\ee
Let us first recall the uniform estimate for $v^\theta$ obtained in \cite{WXX1}.
\begin{pro}\label{uniform-swirl}\cite[Proposition 3.1]{WXX1}
Assume that  $\what{v^\theta}$ is a smooth solution to the problem \eqref{swirl-Fourier}. For every fixed $\xi$, it holds that \be \label{swirl-uniform}
\int_0^1 |\mathcal{L} \what{v^\theta}|^2 r \, dr + \xi^2 \int_0^1 \left| \frac{d}{dr} (r  \what{v^\theta})  \right|^2 \frac1r \, dr + \xi^4 \int_0^1 |\what{v^\theta}|^2 r \, dr
\leq C \int_0^1 |\what{F^\theta}|^2 r \, dr .
\ee
\end{pro}

The next two propositions give some further estimates for $v^\theta$.
\begin{pro}\label{further-estimate-3}
Assume that $\BF^\theta = F^\theta \Be_\theta \in L^2(\Omega)$. The solution $v^\theta$ to the linear problem \eqref{swirlsystem} satisfies that
\be \nonumber
\|\partial_z \Bv^\theta \|_{L^2(\Omega)} \leq C \Phi^{-\frac13} \|\BF^\theta\|_{L^2(\Omega)}.
\ee
\end{pro}

\begin{proof}
 Multiplying the equation in \eqref{swirl-Fourier} by $r \overline{\what{v^\theta}}$ and integrating the resulting equation over  $[0, 1]$ yield
\be \label{4-52}
\int_0^1 i \xi \bar{U}(r) | \what{v^\theta} |^2  r \, dr + \int_0^1 \left| \frac{d}{dr} ( r \what{v^\theta})  \right|^2 \frac1r \, dr + \xi^2 \int_0^1 | \what{v^\theta} |^2 r \, dr
= \int_0^1 \what{F^\theta} \overline{\what{v^\theta}} r \, dr.
\ee
It follows from H\"older inequality, Lemma \ref{lemma1}, and \eqref{4-52}  that one has
\be \label{4-53}
\int_0^1 \left| \frac{d}{dr} ( r \what{v^\theta})  \right|^2 \frac1r \, dr + \xi^2 \int_0^1 | \what{v^\theta} |^2 r \, dr
\leq  \int_0^1 |\what{F^\theta}|^2 r \, dr
\ee
and
\be \label{4-54}
\Phi |\xi| \int_0^1 (1 - r^2)  | \what{v^\theta} |^2  r \, dr \leq C \int_0^1 |\what{F^\theta}| |\what{v^\theta}| r \, dr .
\ee

The equality \eqref{4-54}, together with \eqref{swirl-uniform}, gives
\be \label{4-56} \ba
|\xi|^3 \int_0^1 (1 - r^2) |\widehat{v^\theta}|^2 r \, dr &
\leq C \Phi^{-1} \left( \int_0^1 |\widehat{F^\theta}|^2 r \, dr \right)^{\frac12} \left( \xi^4 \int_0^1 |\widehat{v^\theta}|^2 r \, dr \right)^{\frac12} \\
& \leq C \Phi^{-1} \int_0^1 |\widehat{F^\theta}|^2 r \, dr.
\ea
\ee
By Lemma \ref{weightinequality} and \eqref{4-53}, one has
\be \label{4-56} \ba
|\xi|^2 \int_0^1 (1 - r^2) |\widehat{v^\theta}|^2 r \, dr
& \leq C  \left( |\xi|^3  \int_0^1 (1 - r^2) |\widehat{v^\theta}|^2 r \, dr  \right)^{\frac23} \left( \int_0^1 \left| \frac{d}{dr} ( r \widehat{v^\theta})  \right|^2 \frac1r \, dr \right)^{\frac13} \\
& \leq C \Phi^{-\frac23} \int_0^1 |\widehat{F^\theta}|^2 r \, dr .
\ea
\ee
This implies
\be \label{4-57}
\|\partial_z \Bv^\theta\|_{L^2(\Omega)} \leq C \Phi^{-\frac13} \| \BF^\theta\|_{L^2(\Omega)}.
\ee
Hence the proof of Proposition \ref{further-estimate-3} is completed.
\end{proof}

\begin{pro}\label{further-estimate-4}
Assume that $\BG^\theta = G^\theta \Be_{\theta} \in H^1(\Omega)$. The solution $v^\theta$ to the following linear problem
\be \label{4-61}
\left\{
\ba
& \bar U(r)  \partial_z  v^\theta - \left[ \frac{1}{r} \frac{\partial }{\partial r} \left( r \frac{\partial v^\theta}{\partial r}\right) + \frac{\partial^2 v^\theta}{\partial z^2} - \frac{v^\theta}{r^2} \right] =  \partial_z G^\theta  \ \ \ \mbox{in}\ \ D, \\
& v^\theta(1, z)= v^\theta(0, z) = 0,
\ea
\right.
\ee
satisfies
\be \nonumber
\|\Bv^\theta\|_{L^2(\Omega)} \leq C \Phi^{-\frac13} \|\BG^\theta\|_{L^2(\Omega)}
\ee
and
\be \nonumber
\|\Bv^\theta \|_{H^{\frac{19}{12}} (\Omega)} \leq C \Phi^{-\frac{5}{72}} \|\partial_z \BG^\theta \|_{L^2(\Omega)}^{\frac{19}{24}} \|\BG^\theta \|_{L^2(\Omega)}^{\frac{5}{24}}.
\ee
\end{pro}

\begin{proof}
Similar computations as that in the proof of Proposition \ref{further-estimate-3} yield
\be \label{4-62}
\int_0^1 i \xi \bar{U}(r) |\what{v^\theta}|^2 r \, dr
+ \int_0^1 \left| \frac{d}{dr}( r \what{v^\theta} ) \right|^2  \frac1r \, dr + \xi^2 \int_0^1 |\what{v^\theta}|^2 r \, dr
= i \xi \int_0^1 \what{G^\theta} \overline{\what{v^\theta}} r \, dr .
\ee
It follows from H\"older inequality and Lemma \ref{lemma1} that one has
\be \label{4-63}
\int_0^1 \left|  \frac{d}{dr} (r \what{v^\theta}) \right|^2 \frac1r \, dr + \xi^2 \int_0^1 |\what{v^\theta}|^2 r \, dr
\leq C \int_0^1 |\what{G^\theta}|^2 r \, dr
\ee
and
\be \label{4-64}
\Phi \int_0^1 (1 - r^2) |\what{v^\theta}|^2 r \, dr
\leq C \int_0^1 |\what{G^\theta}|^2 r \, dr.
\ee
Hence, it holds that
\be \label{4-65} \ba
\int_0^1 |\what{v^\theta}|^2 r \, dr &  \leq C \left( \int_0^1 (1  - r^2) |\what{v^\theta} |^2 r \, dr   \right)^{\frac23} \left( \int_0^1 \left| \frac{d}{dr} ( r \what{v^\theta} ) \right|^2 \frac1r \, dr  \right)^{\frac13}  \\
& \leq C \Phi^{-\frac23} \int_0^1 |\what{G^\theta}|^2 r \, dr .
\ea
\ee
This implies
\be \label{4-66}
\| \Bv^\theta \|_{L^2(\Omega)} \leq C \Phi^{-\frac13} \|\BG^\theta\|_{L^2(\Omega)}.
\ee
By the interpolation and Proposition \ref{existence-swirl}, one has
\be \label{4-67} \ba
\|\Bv^\theta \|_{H^{\frac{19}{12}}(\Omega)} & \leq C \|\Bv^\theta\|_{L^2(\Omega)}^{\frac{5}{24}} \|\Bv^\theta\|_{H^2(\Omega)}^{\frac{19}{24}}  \leq C \Phi^{-\frac{5}{72}} \|\BG^\theta \|_{L^2(\Omega)}^{\frac{5}{24}} \|\partial_z \BG^\theta\|_{L^2(\Omega)}^{\frac{19}{24}}.
\ea \ee
Hence the proof of Proposition \ref{further-estimate-4} is finished.
\end{proof}



\section{Existence and Uniqueness of solutions for nonlinear problem}\label{secnonlinear}
In this section, we prove the existence and uniqueness of strong axisymmetric solution of  the nonlinear problem \eqref{NS}-\eqref{fluxBC}, in particular, when $\BF$ and $\Phi$ are large.
The stream function $\psi$ satisfies the following equation
\be \label{ap-40} \ba
& \bar{U}(r) \partial_z (\mathcal{L} + \partial_z^2) \psi - (\mathcal{L} + \partial_z^2)^2 \psi \\
= & \partial_z F^r - \partial_r F^z  - \partial_r (v^r \omega^\theta) - \partial_z (v^z \omega^\theta) + \partial_z \left[ \frac{\left( v^\theta  \right)^2 }{r} \right]  ,
\ea\ee
supplemented with the  boundary condition,
\be \label{ap-41}
\psi(0, z) = \psi(1, z) = \partial_r \psi(1, z) = \mathcal{L}\psi(0, z)= 0 .
\ee
Here  $\omega^\theta = -\partial_r v^z + \partial_z v^r = ( \mathcal{L} + \partial_z^2) \psi$.

The swirl velocity  $\Bv^\theta = v^\theta \Be_\theta $ satisfies the equation
\be \label{ap-42}
\bar{U}(r)  \partial_z \Bv^\theta - \Delta \Bv^\theta = \BF^\theta  - ( v^r \partial_r + v^z \partial_z ) \Bv^\theta
- \frac{v^r}{r} \Bv^\theta
\ee
supplemented
with the  homogeneous boundary condition
\be \label{ap-43}
\Bv^\theta = 0\ \ \ \mbox{on}\ \partial \Omega.
\ee

\begin{proof}[Proof of Theorem \ref{mainresult1}] We divide the proof into three steps.

{\it Step 1. Iteration scheme.} The existence of solutions is proved via an iteration method. Let $\BF^* = F^r \Be_r + F^z \Be_z$. For any given $\BF =  \left(  \BF^* , \BF^\theta  \right) \in  L^2  (\Omega) \times L^2(\Omega)$, there exists a unique solution $(\psi, \Bv^\theta)$ to  the combined linear problem \eqref{vswirl}, \eqref{BC-swirl}, and \eqref{2-0-4-1}--\eqref{2-0-4-4}, and we denote this solution by $ \mathcal{T} \BF$.
Let
\be \nonumber
\BPsi_0
= \mathcal{T} \BmcF, \quad \BPsi_n=(\psi_n , \Bv_n^\theta),
\ee
and
\be \label{ap-50}
\BPsi_{n+1}
= \BPsi_0 + \mathcal{T} (\BF^*_n , \BF^\theta_n)
\ee
with
\be
\BF^*_n = \left[ - v_n^z \omega_n^\theta + \frac{(v_n^\theta)^2 }{r} \right]\Be_r +  v_n^r \omega_n^\theta \Be_z ,\ \ \ \
\BF^\theta_n = -(v_n^r \partial_r + v_n^z \partial_z) \Bv^\theta_n - \frac{v_n^r}{r} \Bv_n^\theta,
   \ee
   where
\be \nonumber
v_i^r =   \partial_z \psi_i , \ \ \ \ v_i^z = -\frac{\partial_r(r \psi_i)}{r} ,\ \ \ \ v_i^\theta = \Bv_i^\theta \cdot \Be_\theta,\ \ \ \ \ \omega_i^\theta = - \partial_r v_i^z + \partial_z v_i^r, \ \  \ i \in \mathbb{N}.
\ee

Set
\be \nonumber
S =  \left\{  (\psi, \Bv^\theta) \in H_*^3(D) \times H^2(\Omega) \left|  \ba & \|\Bv^*\|_{H^{\frac{19}{12}} (\Omega)} \leq 2 C_1 \Phi^{\frac{1}{96}},  \ \ \ \|\Bv^r \|_{L^2(\Omega)} \leq \Phi^{-\frac{15}{32}} \\
& \|\Bv^\theta \|_{H^{\frac{19}{12}}(\Omega)}\leq 2C_2 \Phi^{\frac{1}{96}}, \  \|\partial_z \Bv^\theta\|_{L^2(\Omega)} \leq  \Phi^{-\frac{5}{16}}                 \ea  \right.   \right\}.
\ee
Here $C_1$ and $C_2$ are the two constants indicated in Propositions \ref{existence-stream}-\ref{existence-swirl}.  Under the assumption that $\|\BF\|_{L^2(\Omega)} \leq \Phi^{\frac{1}{96}}$, according to Propositions \ref{existence-stream}-\ref{existence-swirl}, Propositions \ref{further-estimate-1}-\ref{further-estimate-3}, one has $\BPsi_0 \in S$, when $\Phi$ is large enough. Assume that
$\BPsi_n \in S $, our aim is to prove  $\BPsi_{n+1} \in S $.

{\it Step 2. Estimate for the velocity field and existence.}
Denote $\Bv^*_i = v_i^r \Be_r + v_i^z \Be_z $.
It  follows from Sobolev embedding inequalities that the estimates
\be \label{ap-51} \ba
 & \|v_n^r  \omega_n^\theta \|_{L^2(\Omega) }
 \leq C \|v_n^r \|_{L^{12}(\Omega)} \|\omega_n^\theta \|_{L^{\frac{12}{5}}(\Omega)} \leq C \|\Bv_n^r\|_{L^2(\Omega)}^{\frac16} \|\Bv_n^r\|_{L^\infty(\Omega)}^{\frac56}  \|\Bv^*_n\|_{H^{\frac{19}{12}}(\Omega)}   \\
 \leq & C  \|\Bv_n^r\|_{L^2(\Omega)}^{\frac16}\|\Bv^*_n\|_{H^{\frac{19}{12}}(\Omega)}^{\frac{11}{6}}
  \leq   C \Phi^{-\frac{15}{32} \cdot \frac{1}{6}} \Phi^{ \frac{1}{96} \cdot \frac{11}{6}}
\ea \ee
and
\be \label{ap51-1}
 \| v_n^z \omega_n^\theta \|_{L^2(\Omega)} \leq  \|v_n^z \|_{L^\infty(\Omega)} \| \omega_n^\theta\|_{L^2(\Omega)}
\leq C  \|\Bv^*_n\|_{H^{\frac{19}{12}}(\Omega)}^2 \leq  C \Phi^{\frac{1}{48}}
 \ee
 hold.
Moreover, one has
\be \label{ap51-2}
\left\| v_n^\theta \frac{v_n^\theta}{r} \right \|_{L^2(\Omega)}
\leq C \|\Bv_n^\theta\|_{L^\infty(\Omega)} \left\| \frac{\Bv_n^\theta}{r}  \right\|_{L^2(\Omega)}
\leq C \|\Bv_n^\theta \|_{H^{\frac{19}{12}} (\Omega) } \|\Bv_n^\theta \|_{H^1(\Omega)}
\leq C \Phi^{\frac{1}{48}}.
\ee

Similarly, using Sobolev embedding inequalities gives
\be \label{ap-54} \ba
& \| v_n^r \partial_r \Bv_n^\theta\|_{L^2(\Omega)}
\leq C \|v_n^r\|_{L^\infty(\Omega)} \|\partial_r \Bv_n^\theta\|_{L^2(\Omega)}
\leq C \|v_n^r\|_{L^2(\Omega)}^{\frac{1}{19}} \|\Bv^*\|_{H^{\frac{19}{12}}(\Omega)}^{\frac{18}{19} } \|\Bv_n^\theta \|_{H^1(\Omega)} \\
\leq & C \Phi^{-\frac{15}{32} \cdot \frac{1}{19}} \Phi^{\frac{1}{96} \cdot \frac{37}{19}}
\ea \ee
and
\be \label{ap-55} \ba
& \| v_n^z \partial_z \Bv_n^\theta \|_{L^2(\Omega)}
\leq \|v_n^z \|_{L^\infty(\Omega)} \| \partial_z \Bv_n^\theta\|_{L^2(\Omega)}
\leq C \|\Bv^*_n \|_{H^{\frac{19}{12}}(\Omega)} \|\partial_z \Bv_n^\theta\|_{L^2(\Omega)}
\leq  C \Phi^{\frac{1}{96}} \Phi^{-\frac{5}{16}}.
\ea \ee
Furthermore, it holds that
\be \label{ap-56}
\left \|v_n^r \frac{\Bv_n^\theta}{r} \right \|_{L^2(\Omega)}
\leq C \|v_n^r \|_{L^\infty(\Omega)} \left\| \frac{\Bv_n^\theta }{r} \right\|_{L^2(\Omega)}
\leq C \Phi^{-\frac{15}{32}\cdot \frac{1}{19}} \Phi^{\frac{1}{96} \cdot \frac{37}{19}} .
\ee

Combining the estimates \eqref{ap-51}--\eqref{ap-56} and applying Propositions \ref{existence-stream}-\ref{existence-swirl}, \ref{further-estimate-1} yield
\be \label{ap-58}
\| \Bv^*_{n+1} \|_{H^{\frac{19}{12}} (\Omega)} \leq C_1 \Phi^{\frac{1}{96}} + C_3 \left(   \Phi^{-\frac{15}{32}\cdot \frac{1}{6}} \Phi^{ \frac{1}{96} \cdot \frac{11}{6}}
 +  \Phi^{-\frac{1}{30}}\Phi^{\frac{1}{48}} \right)
\ee
and
\be \label{ap-59}
\|\Bv_{n+1}^\theta \|_{H^{\frac{19}{12}}(\Omega)} \leq C_2 \Phi^{\frac{1}{96}} + C_4 \left( \Phi^{-\frac{15}{32}\cdot \frac{1}{19}} \Phi^{\frac{1}{96} \cdot  \frac{37}{19}} +   \Phi^{\frac{1}{96}} \Phi^{-\frac{5}{16}} \right).
\ee
Moreover, it follows from Propositions \ref{further-estimate-1}--\ref{further-estimate-3} that one has
\be \label{ap-60}
\|\Bv_{n+1}^r\|_{L^2(\Omega) } \leq C_5 \max \left\{ \Phi^{-\frac12}, \ \Phi^{-\frac23}   \right\}
\left(  \Phi^{\frac{1}{96}} +   \Phi^{-\frac{15}{32} \cdot \frac{1}{6}} \Phi^{ \frac{1}{96} \cdot \frac{11}{6}}
 +  \Phi^{-\frac{1}{30}}\Phi^{\frac{1}{48}}             \right)
\ee
and
\be \label{ap-61}
\| \partial_z \Bv_{n+1}^\theta \|_{L^2(\Omega)}
\leq C_6 \Phi^{-\frac13} \left( \Phi^{\frac{1}{96}} +  \Phi^{-\frac{15}{32} \cdot \frac{1}{19}} \Phi^{\frac{1}{96} \cdot  \frac{37}{19}} +   \Phi^{\frac{1}{96}} \Phi^{-\frac{5}{16}}                 \right).
\ee

Choose a constant  $\Phi_0$  large enough such that
\be \nonumber
C_3 \left(   \Phi_0^{-\frac{15}{32}\cdot \frac{1}{6}} \Phi_0^{ \frac{1}{96} \cdot \frac{11}{6}}
 +  \Phi_0^{-\frac{1}{30}}\Phi_0^{\frac{1}{48}} \right)< C_1\Phi_0^{\frac{1}{96}}, \ \ \ \  C_4 \left( \Phi_0^{-\frac{15}{32}\cdot \frac{1}{19}} \Phi_0^{\frac{1}{96} \cdot  \frac{37}{19}} +   \Phi_0^{\frac{1}{96}} \Phi_0^{-\frac{5}{16}} \right) < C_2 \Phi_0^{\frac{1}{96}},
\ee
and
\be \nonumber
C_5 \Phi_0^{-\frac12} \left(  \Phi_0^{\frac{1}{96}} +   \Phi_0^{-\frac{15}{32} \cdot \frac{1}{6}} \Phi_0^{ \frac{1}{96} \cdot \frac{11}{6}}
 +  \Phi_0^{-\frac{1}{30}}\Phi_0^{\frac{1}{48}}             \right)  < \Phi_0^{-\frac{15}{32}},
 \ee
 \be \nonumber
 C_6 \Phi_0^{-\frac13} \left( \Phi_0^{\frac{1}{96}} +  \Phi_0^{-\frac{15}{32} \cdot \frac{1}{19}} \Phi_0^{\frac{1}{96} \cdot  \frac{37}{19}} +   \Phi_0^{\frac{1}{96}} \Phi_0^{-\frac{5}{16}}                 \right) < \Phi_0^{-\frac{5}{16}}.
\ee
When $\Phi \geq \Phi_0$, the estimates \eqref{ap-58}--\eqref{ap-61} imply that $(\psi_{n+1}, \Bv_{n+1}^\theta) \in S $.  By mathematical induction,  $\BPsi_n\in S $ for every $n \in \mathbb{N}$.
Note that $\Bv_n = \Bv_n^* + \Bv_n^\theta $. The above estimates show that
\be \label{ap-71}
\|\Bv_n\|_{H^{\frac{19}{12} }(\Omega)} \leq C_0 \Phi^{\frac{1}{96}}\ \ \ \, \text{for every}\,\, n \in \mathbb{N}.
\ee
Since $\{ \Bv_{n} \}$ is uniformly bounded in $H^{\frac{19}{12}}(\Omega)$, there exists a vector-valued function $\Bv \in H^{\frac{19}{12}}(\Omega)$ such that
$ \Bv_{n} \rightharpoonup \Bv $ in $H^{\frac{19}{12}}(\Omega)$ and
\be \nonumber
\|\Bv \|_{H^2(\Omega)} \leq C_0  \Phi^{\frac{1}{96}}.
\ee
Meanwhile, as proved in \cite{WXX1},  the equation \eqref{ap-50} implies that
\be \label{ap-72}
{\rm curl}~\left( (\bBU \cdot \nabla ) \Bv_{n+1} + ( \Bv_{n+1} \cdot \nabla ) \bBU - \Delta \Bv_{n+1} + (\Bv_{n} \cdot \nabla ) \Bv_{n} - \BF      \right) = 0.
\ee
Taking the limit of the equation \eqref{ap-72} yields
\be \label{ap-74}
 {\rm curl}~\left( (\bBU \cdot \nabla ) \Bv + ( \Bv \cdot \nabla ) \bBU - \Delta \Bv  + (\Bv  \cdot \nabla ) \Bv  - \BF      \right) = 0.
\ee
Hence, there exists a function $P$ with $\nabla P \in L^2(\Omega)$,  satisfying
\be \label{ap-73}
 (\bBU \cdot \nabla ) \Bv  + ( \Bv  \cdot \nabla ) \bBU - \Delta \Bv  + (\Bv  \cdot \nabla ) \Bv + \nabla P  = \BF.
\ee
Moreover, according to the regularity estimates in  Propositions \ref{existence-stream}-\eqref{existence-swirl}, it holds that
\be \label{ap-74} \ba
& \|\Bv\|_{H^2(\Omega)} \leq C (1 + \Phi^{\frac14} ) \|\BF\|_{L^2(\Omega)} + C (1 + \Phi^{\frac14}) \|(\Bv \cdot \nabla )\Bv\|_{L^2(\Omega)} \\
\leq  & C  \Phi^{\frac14} \Phi^{\frac{1}{96}} + C \Phi^{\frac14} \|\Bv\|_{H^{\frac{19}{12}} (\Omega)}^2 \leq C \Phi^{\frac{7}{24}}.
\ea
\ee
Let $\Bu = \Bv + \bBU$. Then $\Bu$ is a solution of the problem \eqref{NS}-\eqref{fluxBC}.

{\it Step 3. Uniqueness.} Suppose there are two axisymmetric solutions $\Bu$ and $\tilde{\Bu}$ of the problem \eqref{NS}--\eqref{fluxBC}, satisfying
\be \label{ap-77}
\|\Bu - \bBU \|_{H^{\frac{19}{12}}(\Omega)} \leq C_0 \Phi^{\frac{1}{96}}, \ \ \ \|\tilde{\Bu} - \bBU\|_{H^{\frac{19}{12}}(\Omega)} \leq C_0 \Phi^{\frac{1}{96}},
\ee
and
\be \label{ap-78}
\|\tilde{\Bu}^r \|_{L^2(\Omega)} \leq \Phi^{-\frac{15}{32}}.
\ee
Let
\be \nonumber
\Bv = \Bu - \bBU= v^r \Be_r + v^\theta \Be_\theta + v^z \Be_z, \ \ \ \ \tilde{\Bv} = \tilde{\Bu}- \bBU= \tilde{v}^r \Be_r +
\tilde{v}^\theta \Be_\theta + \tilde{v}^z \Be_z .
\ee
Suppose that $\psi$ and $\tilde{\psi}$  are the stream functions associated with the vector fields $\Bv$ and $\tilde{\Bv}$, respectively. Then $\psi- \tilde{\psi}$ satisfies the following equation,
 \be \nonumber
 \ba
 & \bar{U}(r) \partial_z (\mathcal{L} + \partial_z^2)(\psi - \tilde{\psi} ) - (\mathcal{L} + \partial_z^2)^2 (\psi -\tilde{\psi}) \\
 = & - \partial_r (v^r \omega^\theta - \tilde{v}^r \tilde{\omega}^\theta ) - \partial_z (v^z \omega^\theta - \tilde{v}^z \tilde{\omega}^\theta)
 + 2 \partial_z \left( \frac{v^\theta}{r} v^\theta - \frac{\tilde{v}^\theta}{r}   \tilde{v}^\theta   \right).
 \ea
 \ee
It follows from Sobolev's embedding inequalities that one has
\be \label{ap-81}\ba
& \|v^r \omega^\theta - \tilde{v}^r \tilde{\omega}^\theta \|_{L^2(\Omega)}
\leq \| v^r - \tilde{v}^r  \|_{L^6(\Omega)} \|\omega^\theta\|_{L^3(\Omega)}
+ \| \tilde{v}^r \|_{L^6(\Omega)} \| \omega^\theta - \tilde{\omega}^\theta \|_{L^3(\Omega)} \\
\leq & C \|\Bv^r- \tilde{\Bv}^r \|_{H^{\frac{19}{12}} (\Omega)} \|\Bv^* \|_{H^{\frac{19}{12}} (\Omega)}
+ C \|\tilde{\Bv}^r \|_{L^2(\Omega)}^{\frac13} \|\tilde{\Bv}^r  \|_{H^{\frac{19}{12}} (\Omega)}^{\frac23} \|\Bv^* - \tilde{\Bv}^* \|_{H^{\frac{19}{12}} (\Omega)} \\
\leq & C \|\Bv^r- \tilde{\Bv}^r \|_{H^{\frac{19}{12}} (\Omega)} \Phi^{\frac{1}{96} } + C \|\Bv^* - \tilde{\Bv}^*\|_{H^{\frac{19}{12}}(\Omega)} \Phi^{-\frac{15}{32}\cdot \frac13+ \frac{1}{96} \cdot \frac23 }
\ea
\ee
and
\be \label{ap-82}
\ba
& \|v^z \omega^\theta - \tilde{v}^z \tilde{\omega}^\theta \|_{L^2(\Omega)} \leq
C \|v^z - \tilde{v}^z \|_{L^6(\Omega)} \|\omega^\theta\|_{L^3(\Omega)} + C \|\tilde{v}^z \|_{L^6(\Omega)} \|\omega^\theta - \tilde{\omega}^\theta\|_{L^3(\Omega)} \\
\leq &  C \|\Bv^* - \tilde{\Bv}^* \|_{H^{\frac{19}{12}} (\Omega)} \left( \|\Bv^* \|_{H^{\frac{19}{12}} (\Omega)} + \| \tilde{\Bv}^* \|_{H^{\frac{19}{12}} (\Omega)} \right) \leq C \|\Bv^* - \tilde{\Bv}^* \|_{H^{\frac{19}{12}} (\Omega)} \Phi^{\frac{1}{96}}.
\ea
\ee
Similarly, it holds that
\be \label{ap-83} \ba
& \left\| \frac{v^\theta}{r} v^\theta - \frac{\tilde{v}^\theta}{r} \tilde{v}^\theta  \right\|_{L^2(\Omega)}
\leq C \| \Bv^\theta - \tilde{\Bv}^\theta\|_{H^{\frac{19}{12}}(\Omega)} \left(  \|\Bv^\theta \|_{H^{\frac{19}{12}}(\Omega)} +
\| \tilde{\Bv}^\theta \|_{H^{\frac{19}{12} } (\Omega)}      \right) \\
\leq & C \| \Bv^\theta - \tilde{\Bv}^\theta\|_{H^{\frac{19}{12}}(\Omega)} \Phi^{\frac{1}{96} }.
\ea
\ee

Combining the above estimates \eqref{ap-81}--\eqref{ap-83}, it follows from Propositions \ref{further-estimate-1}--\ref{further-estimate-2} that one has
\be \label{ap-85} \ba
\Phi^{\frac{1}{80}} \|\Bv^r - \tilde{\Bv}^r \|_{H^{\frac{19}{12}} (\Omega)}
\leq  C  \Phi^{\frac{1}{80}} \Phi^{-\frac{1}{40}} &  \left[   \Phi^{\frac{1}{96}} \|\Bv^r- \tilde{\Bv}^r\|_{H^{\frac{19}{12}} (\Omega)}
+ \Phi^{-\frac{43}{96} \cdot  \frac13} \|\Bv^* - \tilde{\Bv}^* \|_{H^{\frac{19}{12}} (\Omega)}   \right. \\
& \left. \ \ \ \ + \Phi^{\frac{1}{96}}  \|\Bv^* - \tilde{\Bv}^* \|_{H^{\frac{19}{12}} (\Omega)} + \Phi^{\frac{1}{96}} \|\Bv^\theta- \tilde{\Bv}^\theta \|_{H^{\frac{19}{12} } (\Omega)}        \right]
\ea \ee
and
\be \label{ap-86} \ba
\|\Bv^z - \tilde{\Bv}^z \|_{H^{\frac{19}{12}} (\Omega)}
\leq &  C  \Phi^{\frac{1}{96}- \frac{1}{80}} \Phi^{\frac{1}{80} } \|\Bv^r - \tilde{\Bv}^r \|_{H^{\frac{19}{12}} (\Omega)} +
C \Phi^{-\frac{43}{96} \cdot \frac13} \|\Bv^* - \tilde{\Bv}^* \|_{H^{\frac{19}{12}} (\Omega)} \\
&\ \ + C \Phi^{-\frac{1}{30}} \Phi^{\frac{1}{96} } \left[ \|\Bv^* - \tilde{\Bv}^* \|_{H^{\frac{19}{12}} (\Omega)}
+ \|\Bv^\theta - \tilde{\Bv}^\theta \|_{H^{\frac{19}{12} } (\Omega)} \right].
\ea
\ee

On the other hand, $\Bv^\theta - \tilde{\Bv}^\theta$ satisfies
\be \nonumber
\ba
& \bar{U}(r) \partial_z (\Bv^\theta - \tilde{\Bv}^\theta) - \Delta (\Bv^\theta - \tilde{\Bv}^\theta)\\
= & - \left( v^r \partial_r \Bv^\theta - \tilde{v}^r \partial_r \tilde{\Bv}^\theta \right) - 2 \left( v^r \frac{\Bv^\theta}{r} - \tilde{v}^r \frac{\tilde{\Bv}^\theta }{r}       \right) -\left(  \partial_r v^r \Bv^\theta - \partial_r \tilde{v}^r \tilde{\Bv}^\theta \right) -  \partial_z
\left(  v^z \Bv^\theta - \tilde{v}^z \tilde{\Bv}^\theta   \right).
\ea
\ee
It follows from Sobolev's embedding inequalities that  one has
\be \label{ap-87} \ba
& \| v^r \partial_r \Bv^\theta - \tilde{v}^r \partial_r \tilde{\Bv}^\theta \|_{L^2(\Omega)}
\leq \|v^r - \tilde{v}^r \|_{L^\infty(\Omega)} \|\partial_r \Bv^\theta \|_{L^2(\Omega)} + \|\tilde{v}^r \|_{L^\infty(\Omega)} \|\partial_r {\Bv}^\theta - \partial_r \tilde{v}^\theta \|_{L^2(\Omega) }\\
\leq & C \Phi^{\frac{1}{96}} \|\Bv^r - \tilde{\Bv}^r \|_{H^{\frac{19}{12}} (\Omega)} + C \|\tilde{\Bv}^r\|_{L^2(\Omega)}^{\frac{1}{19}}
\|\tilde{\Bv}^r \|_{H^{\frac{19}{12}} (\Omega)}^{\frac{18}{19}} \|\Bv^\theta - \tilde{\Bv}^\theta \|_{H^{\frac{19}{12} } (\Omega)} \\
\leq & C \Phi^{\frac{1}{96}- \frac{1}{80} } \Phi^{\frac{1}{80}} \|\Bv^r - \tilde{\Bv}^r \|_{H^{\frac{19}{12}} (\Omega)} +
C \Phi^{-\frac{15}{32} \cdot \frac{1}{19}} \Phi^{\frac{1}{96} \cdot \frac{18}{19} } \|\Bv^\theta - \tilde{\Bv}^\theta \|_{H^{\frac{19}{12} } (\Omega)}.
\ea
\ee
The similar computations give
\be \label{ap-88}
 \left\| v^r \frac{\Bv^\theta}{r} - \tilde{v}^r \frac{\tilde{\Bv}^\theta}{r} \right\|_{L^2(\Omega)} \leq
C \Phi^{\frac{1}{96}- \frac{1}{80} } \Phi^{\frac{1}{80}} \|\Bv^r - \tilde{\Bv}^r \|_{H^{\frac{19}{12}} (\Omega)} +
C \Phi^{-\frac{15}{32} \cdot \frac{1}{19}} \Phi^{\frac{1}{96} \cdot \frac{18}{19} } \|\Bv^\theta - \tilde{\Bv}^\theta \|_{H^{\frac{19}{12} } (\Omega)}
\ee
and
\be \label{ap-89} \ba
& \| \partial_r v^r \Bv^\theta - \partial_r \tilde{v}^r \tilde{\Bv}^\theta\|_{L^2(\Omega)} \\
\leq & C \|\partial_r v^r - \partial_r \tilde{v}^r \|_{L^2(\Omega)} \|\Bv^\theta \|_{L^\infty(\Omega)}
+ C \|\partial_r \tilde{v}^r\|_{L^2(\Omega)} \|\Bv^\theta - \tilde{\Bv}^\theta \|_{L^\infty(\Omega)} \\
\leq & C \|\Bv^r - \tilde{\Bv}^r \|_{H^{\frac{19}{12}} (\Omega)} \|\Bv^\theta \|_{H^{\frac{19}{12}} (\Omega)} + C \|\tilde{\Bv}^r \|_{L^2(\Omega)}^{\frac{7}{19}} \| \tilde{\Bv}^r \|_{H^{\frac{19}{12}} (\Omega)}^{\frac{12}{19} } \|\Bv^\theta - \tilde{\Bv}^\theta \|_{L^\infty(\Omega)}\\
\leq & C \Phi^{\frac{1}{96}- \frac{1}{80} } \Phi^{\frac{1}{80} }\|\Bv^r - \tilde{\Bv}^r \|_{H^{\frac{19}{12}} (\Omega)}
+ C \Phi^{-\frac{15}{32} \cdot \frac{7}{19} + \frac{1}{96} \cdot \frac{12}{19} }  \|\Bv^\theta - \tilde{\Bv}^\theta \|_{H^{\frac{19}{12}}(\Omega)} .
\ea
\ee
Moreover,
\be \label{ap-90} \ba
 & \| v^z \Bv^\theta - \tilde{v}^z \tilde{\Bv}^\theta \|_{L^2(\Omega) }
\leq  C \| \Bv^z - \tilde{\Bv}^z \|_{H^{\frac{19}{12}} (\Omega)}  \|\Bv^\theta \|_{H^{\frac{19}{12}} (\Omega)} +
C \|\tilde{\Bv}^z \|_{H^{\frac{19}{12}} (\Omega)} \| \Bv^\theta - \tilde{\Bv}^\theta \|_{H^{\frac{19}{12}} (\Omega)} \\
\leq & C \Phi^{\frac{1}{96}} \left( \| \Bv^z - \tilde{\Bv}^z \|_{H^{\frac{19}{12}} (\Omega)} + \| \Bv^\theta - \tilde{\Bv}^\theta \|_{H^{\frac{19}{12}} (\Omega)} \right),
\ea
\ee
and
\be \label{ap-91}
\ba
& \|\partial_z (v^z \Bv^\theta ) - \partial_z (\tilde{v}^z \tilde{\Bv}^\theta ) \|_{L^2(\Omega)} \\
\leq & C \Phi^{\frac{1}{96}} \left( \| \Bv^z - \tilde{\Bv}^z \|_{H^{\frac{19}{12}} (\Omega)} + \| \Bv^\theta - \tilde{\Bv}^\theta \|_{H^{\frac{19}{12}} (\Omega)} \right).
\ea
\ee
Combining the estimates \eqref{ap-87}--\eqref{ap-91}, it follows from Propositions \ref{further-estimate-3}--\ref{further-estimate-4} that
\be \label{ap-92} \ba
 \|\Bv^\theta - \tilde{\Bv}^\theta\|_{H^{\frac{19}{12}} (\Omega)}
\leq &   C \Phi^{\frac{1}{96}- \frac{1}{80} } \Phi^{\frac{1}{80}} \|\Bv^r - \tilde{\Bv}^r \|_{H^{\frac{19}{12}} (\Omega)} +
C \Phi^{-\frac{15}{32} \cdot \frac{1}{19}} \Phi^{\frac{1}{96} \cdot \frac{18}{19} } \|\Bv^\theta - \tilde{\Bv}^\theta \|_{H^{\frac{19}{12} } (\Omega)} \\
& \ + C \Phi^{-\frac{101}{608}}  \|\Bv^\theta - \tilde{\Bv}^\theta \|_{H^{\frac{19}{12}}(\Omega)} + C \Phi^{-\frac{17}{288}} \left(  \|\Bv^z - \tilde{\Bv}^z \|_{H^{\frac{19}{12}}(\Omega)}  + \|\Bv^\theta - \tilde{\Bv}^\theta \|_{H^{\frac{19}{12}}(\Omega)}  \right).
\ea
\ee

Combining the three estimates \eqref{ap-85}--\eqref{ap-86} and \eqref{ap-92} together gives  the uniqueness of the solution when $\Phi$ is large enough. Thus the proof of Theorem \ref{mainresult1} is finished.
\end{proof}



\section{Asymptotic Behavior}\label{sec-asymp}
In this section, we investigate the asymptotic behavior of solutions to \eqref{NS}--\eqref{fluxBC} and prove Theorem \ref{mainresult2}. The proof consists of two steps. In the first step, the asymptotic behavior of the solution which is a small perturbation of Hagen-Poiseuille flow is established. In the second step, the smallness requirement is removed since the solution be  a small perturbation of Hagen-Poiseuille flow at far fields when the condition \eqref{condu1} is satisfied.

Before giving the detailed proof of Theorem \ref{mainresult2}, we first state the uniform estimate of $\psi$ for the linear problem \eqref{2-0-4-1}--\eqref{2-0-4-4} when $f\in L_r^2(D)$.
\begin{pro} \label{uniformestimate}
Assume that $f(r, z) \in L_r^2(D)$, the solution $\psi$ obtained in Proposition \ref{existence-stream} belongs to
 $H_*^3(D)$ and satisfies
\be \label{5-101}
\|\Bv^*\|_{H^{\frac53}(\Omega)} \leq C \|f\|_{L_r^2(D)},  \ \ \ \ \ \ \
\|\Bv^*\|_{H^{2}(\Omega)} \leq C (1 + \Phi^{\frac14} ) \|f\|_{L_r^2(D)}
\ee
where the constant $C$ is  independent of $\Phi$.
\end{pro}

\begin{proof}
Let $F(r, z) = - \int_r^1 f(r, z) \, dz.$ Hence, $f(r, z) = \partial_r F(r, z)$. Moreover, for every $(r, z) \in D$, by H\"older inequality,
\be \nonumber
|F(r , z) | \leq \left( \int_0^1 |f(s, z)|^2  s \, ds \right)^{\frac12} | \ln r|^{\frac12},
\ee
which implies that
\be \nonumber
\|F\|_{L_r^2(D)}^2 \leq C \int_{-\infty}^{+\infty} \int_0^1 |f|^2 s \, ds  \cdot \int_0^1 |\ln r| r \, dr \leq C \|f\|_{L_r^2(D)}^2.
\ee
This, together with the regularity estimates in Proposition \ref{existence-stream}, finishes the proof of Proposition \ref{uniformestimate}.
\end{proof}

The following lemma on the estimate between the stream function and the velocity field is needed in the proof of Theorem \ref{mainresult2}.
\begin{lemma}\label{newlemma}
Assume that $\Bv^* = v^r \Be_r + v^z \Be_z \in H^2(\Omega)$ and $\Bv^*$ is axisymmetric. Let $\psi \in H_*^3(D)$ be the corresponding stream function of the velocity field $\Bv^*$. It holds that
\be \label{5-105}
\|\mathcal{L} \psi \|_{L_r^2(D)} + \|\partial_z^2 \psi \|_{L_r^2(D)} + \left\| \frac{1}{r} \frac{\partial}{\partial r}(r \psi )  \right\|_{L_r^2(D)} + \|\partial_z \psi \|_{L_r^2(D)} +
\|\psi \|_{L_r^2(D)} \leq C \|\Bv^*\|_{H^1(\Omega)}
\ee
and
\be \label{5-106}
\| \psi \Be_r\|_{H^2(\Omega)} \leq C \|\Bv^*\|_{H^1(\Omega)}.
\ee
\end{lemma}
\begin{proof}
Recall that the stream function  $\psi$ of the vector field $\Bv^*$ satisfies
\be \label{5-107}
(\mathcal{L} + \partial_z^2) \psi = \partial_z v^r - \partial_r v^z  \ \ \ \ \mbox{in}\ D.
\ee
Multiplying \eqref{5-107} by $(\mathcal{L} + \partial_z^2) \psi r $ and integrating over $D$, integration by parts gives
\be \label{5-108} \ba
& \int_{-\infty}^{+\infty} \int_0^1 \left[ |\mathcal{L} \psi |^2 r  + 2 |\partial_r \partial_z (r \psi )|^2 \frac1r + |\partial_z^2 \psi |^2 r  \right] \, dr dz \\
 \leq &  C \int_{-\infty}^{+\infty} \int_0^1 |\partial_z v^r - \partial_r v^z |^2 r \, dr dz
 \leq C \|\Bv^*\|_{H^1(\Omega)}^2 .
\ea \ee
Furthermore, by Lemma \ref{lemma1} and \eqref{5-108}, one has
\be  \label{5-109} \ba
 & \left\| \frac{1}{r} \frac{\partial}{\partial r}(r \psi )  \right\|_{L_r^2(D)} + \|\partial_z \psi \|_{L_r^2(D)} +
\|\psi \|_{L_r^2(D)} \\
\leq & C \|\mathcal{L} \psi \|_{L_r^2(D)} + C \left\| \frac1r \frac{\partial^2}{\partial r \partial z} (r \psi ) \right\|_{L_r^2(D)}  \leq C \|\Bv^*\|_{H^1(\Omega)}.
\ea
\ee
The straightforward computations yield
\be \nonumber
\left\{
\ba
& \Delta (\psi \Be_r ) = (\mathcal{L} + \partial_z^2) \psi \Be_r , \ \ \ \mbox{in}\ \Omega, \\
& \psi = 0,\ \ \ \ \mbox{on}\ \partial \Omega.
\ea
\right.
\ee
It follows from the regularity theory for elliptic equations \cite{GT} and \eqref{5-108} that one has
\be \label{5-111}
\|\psi \Be_r\|_{H^2(\Omega)} \leq C \| (\mathcal{L} + \partial_z^2) \psi \Be_r \|_{L^2(\Omega)} \leq C \|\Bv^*\|_{H^1(\Omega)}.
\ee
The proof of Lemma \ref{newlemma} is completed.
\end{proof}

\begin{pro} \label{mainresult3} Assume that $\BF \in L^2(\Omega)$, $\BF= \BF(r, z)$ is axisymmetric, and $\Bu \in H^2(\Omega)$ is an axisymmetric solution to the problem \eqref{NS}--\eqref{fluxBC}.
   There exist a constant $\epsilon_0$, independent of $\BF$ and $\Phi$, and a constant $\alpha_0(\leq 1)$ depending on $\Phi$,  such that if
\be \label{condition5}
\|e^{\alpha |z|} \BF\|_{L^2(\Omega)} < + \infty\ \ \ \
\ee
{with some} $\alpha \in (0, \alpha_0)$ and
\be \label{condition5-1}
\|\Bu - \bBU\|_{H^{\frac53}(\Omega)} \leq \epsilon_0,
\ee
then the solution $\Bu$ satisfies
\be \label{asy0}
\|e^{\alpha |z| } ( \Bu - \bBU) \|_{H^{\frac53}(\Omega)} \leq C \|e^{\alpha |z| }\BF\|_{L^2(\Omega)}.
\ee
Here $C$ is a uniform constant independent of $\BF$ and $\Phi$.
\end{pro}

\begin{proof}
Let $\Bv = \Bu - \bBU$ and $\psi$ be the stream function associated with the vector field $\Bv$. Then $(\psi, \Bv^\theta)$ satisfies the problem \eqref{ap-40}--\eqref{ap-43}. Multiplying \eqref{ap-40} and \eqref{ap-42} by $e^{\alpha z}$ gives
\be \label{asy1} \ba
& \bar{U}(r)   \partial_z (\mathcal{L} + \partial_z^2) (e^{\alpha z} \psi) - (\mathcal{L} + \partial_z^2)^2 (e^{\alpha z} \psi) \\
= & \partial_z(e^{\alpha z} F^r) - \partial_r (e^{\alpha z} F^z )  - \partial_r \left[ v^r (\mathcal{L} + \partial_z^2) (e^{\alpha z} \psi ) \right]- \partial_z \left[ v^z (\mathcal{L} + \partial_z^2) (e^{\alpha z}\psi )  \right] \\
& + \partial_z \left[ \frac{v^\theta}{r} (e^{\alpha z} v^\theta)  \right] + \bar{U}(r) \left[\alpha \mathcal{L} ( e^{\alpha z} \psi ) + \alpha^3 e^{\alpha z} \psi - 3\alpha^2 \partial_z (e^{\alpha z} \psi) + 3\alpha \partial_z^2 (e^{\alpha z} \psi )\right]\\
& - \partial_z \left[  4\alpha \partial_z^2 (e^{\alpha z}\psi) - 6 \alpha^2 \partial_z (e^{\alpha z}\psi) + 4\alpha^3  (e^{\alpha z } \psi) \right] +  \alpha^4 e^{\alpha z}\psi \\
&- \left[ 4\alpha \mathcal{L} \partial_z (e^{\alpha z} \psi)- 2 \alpha^2 \mathcal{L} (e^{\alpha z}\psi) \right] - \alpha e^{\alpha z} F^r \\
 & + \partial_r \left\{ v^r \left[  2 \alpha \partial_z (e^{\alpha z} \psi ) - \alpha^2 e^{\alpha z}\psi \right] \right\}
 + \partial_z \left\{ v^z \left[  2 \alpha \partial_z (e^{\alpha z} \psi ) - \alpha^2 e^{\alpha z} \psi         \right]  \right\}\\
 & +  \alpha v^z (\mathcal{L} + \partial_z^2) (e^{\alpha z} \psi ) - \alpha v^z \left[ 2\alpha \partial_z (e^{\alpha z} \psi ) - \alpha^2 e^{\alpha z} \psi   \right] - \alpha \frac{v^\theta}{r} (e^{\alpha z} v^\theta) ,
\ea \ee
and
\be \label{asy2}
\ba
& \bar{U} (r)  \partial_z(e^{\alpha z} \Bv^\theta) - \Delta ( e^{\alpha z} \Bv^\theta)
 =e^{\alpha z} \BF^\theta - (\Bv^* \cdot \nabla) (e^{\alpha z} \Bv^\theta) - \frac{v^r}{r}  (e^{\alpha z} \Bv^\theta) \\[2mm]
&\quad \quad \ \ \ \ + \bar{ U}(r)   \alpha e^{\alpha z} \Bv^\theta - 2 \alpha \partial_z(e^{\alpha z} \Bv^\theta) +\alpha^2 e^{\alpha z} \Bv^\theta
+ v^z \alpha e^{\alpha z} \Bv^\theta.
\ea
\ee

Denote
\be \nonumber
v_\alpha^r = \partial_z (e^{\alpha z} \psi),\ \ \ v_\alpha^z = -\frac{\partial_r (r e^{\alpha z} \psi )}{r},\ \ \ \Bv_\alpha^* = v_\alpha^r \Be_r + v_\alpha^z \Be_z,\ \ \Bv_\alpha^\theta = e^{\alpha z} v^\theta \Be_\theta.
\ee
Regarding the terms on the right hand of \eqref{asy1}, by Sobolev embedding inequalities and Lemma \ref{newlemma}, one has
\be \label{asy5} \ba
& \| v^r ( \mathcal{L} + \partial_z^2) (e^{\alpha z} \psi ) \|_{L_r^2(D)}
\leq C \|v^r \|_{L^\infty(\Omega)} \|(\mathcal{L} + \partial_z^2) (e^{\alpha z} \psi ) \|_{L_r^2(D)} \\
\leq & C \|\Bv\|_{H^{\frac53}(\Omega)} \|\Bv_\alpha^* \|_{H^1(\Omega)}
\leq C \|\Bv\|_{H^{\frac53} (\Omega) }  \|\Bv_\alpha^* \|_{H^{\frac53}(\Omega)}
\ea \ee
and
\be \label{asy6} \ba
& \| v^z  ( \mathcal{L} + \partial_z^2) (e^{\alpha z} \psi ) \|_{L_r^2(D)}
\leq C \|v^z \|_{L^\infty(\Omega)} \|(\mathcal{L} + \partial_z^2) (e^{\alpha z} \psi ) \|_{L_r^2(D)} \\
\leq & C \|\Bv\|_{H^{\frac53}(\Omega)} \|\Bv_\alpha^* \|_{H^1(\Omega)}
\leq C \|\Bv\|_{H^{\frac53} (\Omega) }  \|\Bv_\alpha^* \|_{H^{\frac53}(\Omega)} .
\ea
\ee
Similarly, one has
 \be \label{asy6-1} \ba
& \left\| \frac{v^\theta}{r} e^{\alpha z} v^\theta  \right\|_{L_r^2(D)}
\leq \left\| \frac{v^\theta }{r} \right\|_{L^2(\Omega)} \|e^{\alpha z} v^\theta \|_{L^{\infty} (\Omega)}
\leq C \|\Bv \|_{H^{\frac53} (\Omega)} \|\Bv_\alpha^\theta \|_{H^{\frac53}(\Omega)} .
\ea \ee
Note that
\be \label{asy7} \ba
& \| \bar{U}(r) \left[\alpha \mathcal{L} ( e^{\alpha z} \psi ) + \alpha^3 e^{\alpha z} \psi - 3\alpha^2 \partial_z (e^{\alpha z} \psi) + 3\alpha \partial_z^2 (e^{\alpha z} \psi )\right]  \|_{L_r^2(D)} \\
\leq & C \Phi |\alpha| \| \Bv_\alpha^* \|_{H^1(\Omega)} \leq C \Phi |\alpha| \| \Bv_\alpha^* \|_{H^{\frac53}(\Omega)}.
\ea \ee
and
\be \label{asy8} \ba
&  \|  4\alpha \partial_z^2 (e^{\alpha z}\psi) - 6 \alpha^2 \partial_z (e^{\alpha z}\psi) + 4\alpha^3 (e^{\alpha z } \psi) \|_{L_r^2(D)} + \| \alpha^4 e^{\alpha z}\psi  \|_{L_r^2(D) } \\
&\quad + \| 4 \alpha \mathcal{L} (e^{\alpha z} \psi ) \|_{L_r^2(D)} + \|2\alpha^2 \mathcal{L} (e^{\alpha z} \psi ) \|_{L_r^2(D)}\\
 \leq &  C |\alpha | \| \Bv_\alpha^* \|_{H^1(\Omega)} + C |\alpha| \| \Bv_\alpha^* \|_{H^{\frac53}(\Omega)} \leq C |\alpha| \| \Bv_\alpha^* \|_{H^{\frac53}(\Omega)}.
\ea
\ee
Furthermore, it holds that
\be \label{asy10} \ba
&\| v^r \left[   2\alpha \partial_z (e^{\alpha z} \psi ) - \alpha^2 e^{\alpha z} \psi   \right]  \|_{L_r^2(D)} + \| v^z \left[   2\alpha \partial_z (e^{\alpha z} \psi ) - \alpha^2 e^{\alpha z} \psi   \right]  \|_{L_r^2(D)}\\
&\quad +\| \alpha v^z (\mathcal{L} + \partial_z^2) (e^{\alpha z} \psi ) - \alpha v^z \left[ 2\alpha \partial_z (e^{\alpha z} \psi ) - \alpha^2 e^{\alpha z} \psi  \right] \|_{L_r^2(D)}\\
\leq & C |\alpha | \|v^r \|_{L^\infty(\Omega)} \| \Bv_\alpha^* \|_{H^{\frac53}(\Omega)} +  C |\alpha | \|\Bv\|_{H^{\frac53} (\Omega)} \| \Bv_\alpha^* \|_{H^{\frac53}(\Omega)}\\
\leq & C |\alpha | \|\Bv\|_{H^{\frac53} (\Omega)} \| \Bv_\alpha^* \|_{H^{\frac53}(\Omega)}.
\ea
\ee
Similar to \eqref{asy6-1}, one has
\be \label{asy10-3}
\left\| \alpha \frac{v^\theta}{r} (e^{\alpha z} v^\theta) \right\|_{L_r^2(D)}
\leq C |\alpha | \|\Bv\|_{H^{\frac53} (\Omega)} \| \Bv_\alpha^\theta \|_{H^{\frac53}(\Omega)}.
\ee

It follows from the  Sobolev embedding inequalities that one has
\be \label{asy15}
\| ( \Bv^* \cdot \nabla ) (e^{\alpha z} \Bv^\theta  ) \|_{L^2(\Omega)}
\leq C \|\Bv^* \|_{L^\infty(\Omega)}  \|e^{\alpha z} \Bv^\theta  \|_{H^1(\Omega)}
\leq C \|\Bv\|_{H^{\frac53} (\Omega)} \| \Bv_\alpha^\theta \|_{H^{\frac53}(\Omega)}
\ee
and
\be \label{asy16}  \left\| \frac{v^r}{r} (e^{\alpha z} \Bv^\theta )   \right\|_{L^2(\Omega)}
\leq C \|\partial_r v^r + \partial_z v^z \|_{L^3(\Omega)} \|e^{\alpha z} \Bv^\theta \|_{L^6(\Omega)}
\leq  C  \|\Bv\|_{H^{\frac53} (\Omega)} \| \Bv_\alpha^\theta \|_{H^{\frac53}(\Omega)}.
\ee
Furthermore, it holds that
\be \label{asy17}
\| -2 \alpha \partial_z (e^{\alpha z} \Bv^\theta ) + \alpha^2 e^{\alpha z} \Bv^\theta \|_{L^2(\Omega)}
\leq C |\alpha | \| \Bv_\alpha^\theta\|_{H^{\frac53}(\Omega)}
\ee
and
\be \label{asy18} \ba
& \|\bar{U}(r) \alpha e^{\alpha z} \Bv^\theta \|_{L^2(\Omega)} + \|v^z \alpha e^{\alpha z} \Bv^\theta \|_{L^2(\Omega)} \\
\leq & C |\alpha| \Phi \|e^{\alpha z} \Bv^\theta \|_{H^{\frac53}(\Omega)} + C |\alpha| \|v^z\|_{L^\infty(\Omega)} \|e^{\alpha z} \Bv^\theta \|_{H^{\frac53}(\Omega)}\\
\leq & C |\alpha| \left( \Phi + \|\Bv\|_{H^{\frac53} (\Omega)} \right)  \|  \Bv^\theta_\alpha \|_{H^{\frac53}(\Omega)}.
\ea \ee

Hence, collecting the estimates \eqref{asy5}--\eqref{asy18} and applying Propositions \ref{existence-stream}-\ref{existence-swirl}, Proposition \ref{uniformestimate} yield
\be \label{asy19} \ba
& \| \Bv_\alpha^* \|_{H^{\frac53} (\Omega) } + \| \Bv^\theta_\alpha \|_{H^{\frac53}(\Omega)} \\
\leq & C \|e^{\alpha z} \BF\|_{L^2(\Omega)} + C_7 \|\Bv \|_{H^{\frac53}(\Omega)} \left( \|\Bv_\alpha^*\|_{H^{\frac53}(\Omega)} +
\|\Bv^\theta_\alpha\|_{H^{\frac53}(\Omega)}  \right) \\
&\ \ \ + C_8 |\alpha| \left(\Phi + 1 + \|\Bv\|_{H^{\frac53}(\Omega)}\right) \left( \|\Bv_\alpha^*\|_{H^{\frac53}(\Omega)} +
\|\Bv^\theta_\alpha\|_{H^{\frac53}(\Omega)}  \right) .
\ea \ee
Choose a small constant $\epsilon_0$ such that
$
C_7 \epsilon_0 \leq \frac14.
$
If $\Bu$ and $\alpha $ satisfy
\be \nonumber
\|\Bv\|_{H^{\frac53}(\Omega) } = \|\Bu - \bBU \|_{H^{\frac53}(\Omega)} \leq \epsilon_0\ \ \ \ \mbox{and}\ \ \ |\alpha| \leq \alpha_0 \leq \frac{1}{2C_8 (1+ \Phi+ \epsilon_0) },
\ee
then the inequality \eqref{asy19} implies that
\be \label{asy20}
\| \Bv_\alpha^* \|_{H^{\frac53} (\Omega)} + \|\Bv_\alpha^\theta\|_{H^{\frac53}(\Omega)}
\leq C  \|e^{\alpha z}\BF\|_{L^2(\Omega)}.
\ee

 Note that
\be \nonumber
e^{\alpha z} v^r =  \partial_z(e^{\alpha z} \psi) - \alpha e^{\alpha z}\psi = v_\alpha^r - \alpha e^{\alpha z}\psi \ \ \ \mbox{and}\ \ \ e^{\alpha z} v^z = -  \frac{\partial_r ( r e^{\alpha z} \psi)}{\partial r}= v_\alpha^z.
\ee
Decompose $e^{\alpha z} \Bv^*$ into two parts as follows
\be \nonumber
e^{\alpha z}\Bv^* = \Bv_\alpha^* - \alpha e^{\alpha z}\psi \Be_r.
\ee
It follows from Lemma \ref{newlemma} that one has
\be \label{asy21}
\|e^{\alpha z} \psi \Be_r\|_{H^{\frac53}(\Omega)}
\leq \|e^{\alpha z} \psi \Be_r\|_{H^2(\Omega)} \leq C \|\Bv_\alpha^* \|_{H^1(\Omega)} \leq C \|\Bv_\alpha^*\|_{H^{\frac53}(\Omega)}.
\ee
This, together with \eqref{asy20}, implies
\be
\|e^{\alpha z } ( \Bu - \bBU) \|_{H^{\frac53}(\Omega\cap \{z\geq 0\})} \leq C \|e^{\alpha |z| }\BF\|_{L^2(\Omega)}.
\ee
Similarly, one has
\be
\|e^{-\alpha z } ( \Bu - \bBU) \|_{H^{\frac53}(\Omega\cap \{z\leq 0\})} \leq C \|e^{\alpha |z| }\BF\|_{L^2(\Omega)}.
\ee
Hence the proof the proof of Proposition \ref{mainresult3} is completed.
\end{proof}

Next, we remove the smallness requirement for $\Bu - \bBU $ in Proposition \ref{mainresult3}. The key observation is that $\|\Bu-\bBU\|_{H^2(\Omega_L)}$ with $\Omega_L=B_1(0)\times (L, \infty)$ is sufficiently small for sufficiently large $L$, provided that $\|\Bu-\bBU\|_{H^2(\Omega)}$ is bounded. This implies that $\Bu$ satisfies the assumptions of Proposition \ref{mainresult3} in the domain $\Omega_L$, and hence Theorem \ref{mainresult2} can be proved in the similar way.

\begin{proof}[Proof of Theorem \ref{mainresult2}]  Let $\Bv = \Bu - \bBU$ and $\psi$ be the corresponding stream function.
Let $L$  be a positive constant to be determined. Denote
  $\Omega_{L} = B_1(0) \times (L, +\infty)$.
Choose a smooth cut-off function $\eta(z)$ satisfying
\be \nonumber
\eta(z)= \left\{ \begin{array}{l} 0,\ \ \ z\leq L,\\
1, \ \ \ z\geq L + 1.   \end{array}  \right.
\ee
Note that $(\psi, \Bv^\theta)$ is a solution to the problem \eqref{ap-40}--\eqref{ap-43}.
Multiplying \eqref{ap-40} by $\eta(z)$ yields
\be \label{asy-87} \ba
& \bar{U}(r) \partial_z (\mathcal{L} + \partial_z^2)(\eta \psi ) - (\mathcal{L} + \partial_z^2)^2 (\eta \psi) \\
=   & \partial_z (\eta F^r + \tilde{F}^r) - \partial_r (\eta F^z + \tilde{F}^z ) + \tilde{f}
- \partial_r \left[ v^r (\mathcal{L} + \partial_z^2) (\eta \psi) \right] \\
&\ \ \ \ \ \ - \partial_z \left[ v^r( \mathcal{L} + \partial_z^2)(\eta \psi)        \right] + \partial_z \left[ \frac{v^\theta}{r} \eta v^\theta     \right],
\ea \ee
where
\be \nonumber  \ba
\tilde{F}^r = &   \bar{U}(r) \left[  2 \eta^{\prime}(z) \partial_z \psi + \eta^{\prime \prime}(z) \psi    \right]  - 4 \eta^{\prime}(z) \mathcal{L} \psi - 4 \eta^{\prime}(z) \partial_z^2 \psi + 2 v^z \eta^{\prime}(z) \partial_z \psi + v^z \eta^{\prime \prime}(z) \psi  ,
\ea  \ee
\be \nonumber
\tilde{F}^z = - \left[ 2v^r \eta^{\prime}(z) \partial_z \psi + v^r \eta^{\prime \prime} (z) \psi \right] ,
\ee
and
\be \nonumber \ba
\tilde{f} = &  - \eta^{\prime} (z) F^r + \bar{U}(r) \eta^{\prime}(z) ( \mathcal{L} + \partial_z^2) \psi + 2\eta^{\prime \prime}(z) \mathcal{L} \psi  - \left[ \eta^{(4)}(z) \psi + 4\eta^{(3)}(z) \partial_z \psi + 2 \eta^{\prime \prime}(z) \partial_z^2 \psi            \right] \\
 &\ \ \ \  + v^z \eta^{\prime}(z) (\mathcal{L} + \partial_z^2)\psi - \frac{v^\theta}{r} \eta^{\prime}(z) v^\theta .
\ea \ee
Similarly, one has
\be \label{asy-88}
\bar{U}(r) \partial_z(\eta \Bv^\theta)  - \Delta (\eta \Bv^\theta) = \eta \BF^\theta + \tilde{\BF}^\theta - (v^r \partial_r + v^z \partial_z )(\eta \Bv^\theta) - \frac{v^r}{r} (\eta \Bv^\theta ),
\ee
where
\be \nonumber
\tilde{\BF}^\theta = \bar{U}(r) \eta^{\prime}(z)  \Bv^\theta - \left[ 2 \eta^{\prime}(z) \partial_z \Bv^\theta + \eta^{\prime \prime} (z) \Bv^\theta  \right] + v^z \eta^{\prime}(z) \Bv^\theta.
\ee

Denote
\be  \nonumber
v_{\alpha, \eta}^r = \partial_z (e^{\alpha z} \eta \psi ) , \ \ \ v_{\alpha, \eta}^z = - \frac{\partial_r (r e^{\alpha z} \eta \psi )}{r}, \,\,  \text{and}\,\,
\Bv_{\alpha, \eta}^{*}= v_{\alpha, \eta}\Be_r + v_{\alpha, \eta}\Be_z.
\ee
It follows from the same lines as in the proof of Proposition \ref{mainresult3} that one has
\be  \la{asy-89}
\ba
& \| \Bv^*_{\alpha, \eta}  \|_{H^{2}(\Omega)} + \|e^{\alpha z} (\eta \Bv^\theta) \|_{H^{2}(\Omega)} \\
\leq & C (1+\Phi^{\frac14}) ( \| e^{\alpha z} \eta \BF\|_{L^2(\Omega)} + \|e^{\alpha z} \tilde{F}^r \|_{L^2_r(D)} + \| e^{\alpha z} \tilde{F}^z \|_{L_r^2(D)} + \| e^{\alpha z} \tilde{f}\|_{L_r^2(D)}) \\
&\ \ \ \ +
C_9 (1 + \Phi^{\frac14}) \| \Bv \|_{H^{\frac53}(\Omega_L)} \left(  \| \Bv^*_{\alpha, \eta}  \|_{H^{2}(\Omega)} + \|e^{\alpha z} (\eta \Bv^\theta) \|_{H^{2}(\Omega)}            \right) \\
&\ \ \ \ \ \ + C_{10} (1 + \Phi^{\frac14} ) |\alpha| \left( \Phi +  1  +   \| \Bv \|_{H^{\frac53}(\Omega_L)} \right) \left(  \| \Bv^*_{\alpha, \eta}  \|_{H^{2}(\Omega)} + \|e^{\alpha z} (\eta \Bv^\theta) \|_{H^{2}(\Omega)}            \right).
\ea
\ee

Since
 $\|\Bv  \|_{H^{\frac53}(\Omega)} < + \infty$,
there exists an $L$ large enough such that
\be \nonumber
  \|\Bv \|_{H^{\frac53}(\Omega_{L})} \leq \min \left\{ \frac{1}{4C_9(1 + \Phi^{\frac14} ) }, \ 1       \right \}.
\ee
Choose an $\alpha_0>0$ small enough such that
\be \nonumber
C_{10} (1 + \Phi^{\frac14} ) ( \Phi+ 2) \alpha_0 \leq  \frac14.
\ee
For every $\alpha \in ( 0, \alpha_0)$, it holds that
\be \nonumber
\ba
& \|\Bv_{\alpha, \eta}^* \|_{H^{2}(\Omega)} + \|e^{\alpha z} (\eta \Bv^\theta) \|_{H^{2}(\Omega)} \\
\leq & C (1 + \Phi^{\frac14} ) \left( \|e^{\alpha z} \tilde{F}^r\|_{L^2_r(D)}
+ \|e^{\alpha z} \tilde{F}^z \|_{L_r^2(D)} + \|e^{\alpha z} \tilde{f} \|_{L^2_r(D)}    + \| e^{\alpha z} \tilde{\BF}^\theta \|_{L^2(\Omega)} \right)  \\
& \ \ \ \ +  C (1 + \Phi^{\frac14} ) \|e^{\alpha z} \eta \BF \|_{L^2(\Omega)}.
\ea
\ee
Note that $ \supp \tilde{F}^r$, $\supp \tilde{F}^z$, $\supp \tilde{f}$, $\supp \tilde{\BF}^\theta  \subseteq  \overline{B_1^2(0)} \times [L, L+ 1]$,  then
\be
\| e^{\alpha z} \tilde{F}^r \|_{L_r^2(D)} + \|e^{\alpha z} \tilde{F}^z \|_{L_r^2(D)} + \|e^{\alpha z}  \tilde{f} \|_{L_r^2(D)} +  \|e^{\alpha z} \tilde{\BF}^\theta \|_{L^2(\Omega)}  \leq C (1 + \Phi) \| \Bv \|_{H^1(\Omega)}.
\ee
Hence, one has
\be \label{asy-93}
\|\Bv_{\alpha, \eta}^* \|_{H^{2}(\Omega)} + \|e^{\alpha z} (\eta \Bv^\theta) \|_{H^{2}(\Omega)} < + \infty.
\ee

Similar to the proof of Proposition \ref{mainresult3}, one can rewrite  $e^{\alpha z} \eta \Bv^*$ as follows,
\be \nonumber
e^{\alpha z} \eta \Bv^* = \Bv_{\alpha, \eta }^*  - \left( \alpha e^{\alpha z} \eta \psi + e^{\alpha z} \eta^{\prime} (z) \psi    \right)\Be_r.
\ee
Thus it follows from Lemma \ref{newlemma} that one has
\be \label{asy-95}
\|( \alpha e^{\alpha z} \eta \psi ) \Be_r   \|_{H^2(\Omega  )} + \| e^{\alpha z} \eta^{\prime} (z) \psi \Be_r \|_{H^2(\Omega)}
\leq C \| \Bv_{\alpha, \eta}^* \|_{H^1(\Omega)} + C \|\Bv\|_{H^1(\Omega)} < + \infty.
\ee
This, together with \eqref{asy-93}, completes   the proof of Theorem \ref{mainresult2}.
\end{proof}



\appendix

\section{Some elementary lemmas}\label{Appendix}
In this appendix, we collect some basic lemmas which play important roles in the paper. Their proofs can be found in \cite[Appendix A]{WXX1}, so we omit the
details here.

The following lemma is about Poincar\'e type inequalities.
\begin{lemma}\label{lemma1}
For a  function $g\in C^2([0,1])$
it holds that
\be \label{2-1-11}
\int_0^1 |g |^2  r \, dr \leq  \int_0^1 \left|   \frac{d}{dr} (r g )     \right|^2 \frac1r \, dr.
\ee
If, in addition, $g(0)=g(1)=0$, then one has
\be
  \int_0^1 \left|   \frac{d}{dr} (r g )     \right|^2 \frac1r \, dr
  \leq \left( \int_0^1 | \mathcal{L}g|^2  r \, dr  \right)^{\frac12} \left( \int_0^1 |g|^2 r\, dr  \right)^{\frac12} \leq \int_0^1 | \mathcal{L}g|^2  r \, dr.
\ee
\end{lemma}



The following lemma is a variant of Hardy-Littlewood-P\'olya type inequality \cite{HLP}, which plays an important role in many estimates in the paper.
\begin{lemma}\label{lemmaHLP}
Let $g\in C^1([0,1])$ satisfy $g(0)=0$, one has
\begin{equation}\label{ineqHLP}
\int_0^1|g(r)|^2 dr \leq \frac{1}{2} \int_0^1 |g^{\prime}(r)|^2 (1-r^2) \, dr,
\end{equation}
and
\be \label{HLP-2}
\int_0^1 |g|^2 r \, dr \leq C \int_0^1 \left| \frac{d(r g) }{dr}   \right|^2 \frac{1-r^2}{r} \, dr .
\ee
\end{lemma}

The following lemma is about a weighted interpolation inequality, which is quite similar to \cite[(3.28)]{M}.
\begin{lemma}\label{weightinequality} Let $g \in C^2[0, 1]$, then one has
\begin{equation} \label{weight1} \ba
\int_0^1 |g|^2r \, dr  \leq & C \left(\int_0^1 (1-r^2)|g|^2 r \, dr\right)^{\frac23} \left(\int_0^1 \frac{1}{r} \left|\frac{d}{dr}(rg)\right|^2 \, dr\right)^{\frac13} \\
&\ \ \ \  + C \int_0^1 (1-r^2)|g|^2 r\, dr ,
\ea
\end{equation}
and
\be \label{weight2} \ba
\int_0^1 \left| \frac{d}{dr}(rg) \right|^2 \frac1r \, dr &  \leq C \left[ \int_0^1 \frac{1-r^2}{r} \left| \frac{d}{dr} ( rg)  \right|^2 \, dr \right]^{\frac23} \left( \int_0^1 |\mathcal{L} g|^2 r \, dr  \right)^{\frac13} \\
&\ \ \ \ \ \ \ \ \ \ \ + C \int_0^1 \frac{1-r^2}{r} \left| \frac{d}{dr} ( rg)  \right|^2 \, dr .
\ea
\ee
\end{lemma}




{\bf Acknowledgement.}
The research of Wang was partially supported by NSFC grant 11671289. The research of  Xie was partially supported by  NSFC grants 11971307 and 11631008,  and Young Changjiang Scholar of Ministry of Education in China. The authors would like to thank Professor Yasunori Maekawa for helpful discussions.


\begin{thebibliography}{99}
\bibitem{ADN}S. Agmon, A. Douglis and L. Nirenberg, Estimates near the boundary for solutions of elliptic partial differential equations satisfying general boundary conditions. II., {\it Commun. Pure Appl. Math.}, {\bf 17} (1964), 35-92.

\bibitem{AP}
K. A. Ames and L. E. Payne,  Decay estimates in steady pipe flow, {\it SIAM J. Math. Anal.}, {\bf 20} (1989), 789--815.

\bibitem{Amick1}
C. J. Amick, Steady solutions of the Navier-Stokes equations in unbounded channels and pipes, {\it  Ann. Scuola Norm. Sup. Pisa Cl. Sci.}, {\bf 4} (1977), 473--513.

\bibitem{Amick2}
C. J. Amick,  Properties of steady Navier-Stokes solutions for certain unbounded channels and pipes, {\it Nonlinear Anal.}, {\bf 2} (1978), 689--720.

\bibitem{AmickF}
C. J. Amick and L. E. Fraenkel, Steady solutions of the Navier-Stokes equations representing plane flow in channels of various types, {\it Acta Math.}, {\bf 144} (1980), 83--151.


\bibitem{Galdi}
G. P. Galdi, An introduction to the mathematical theory of the Navier-Stokes equations. Steady-state problems, Second edition, Springer Monographs in Mathematics. Springer, New York, 2011.

\bibitem{M} I. Gallagher, M. Higaki and Y. Maekawa, On stationary two-dimensional flows around a fast rotating disk, {\it Math. Nachr.}, {\bf 292} (2019), 273--308.

\bibitem{GT}
D. Gilbarg and N. Trudinger,
{\it  Elliptic Partial Differential Equations of Second Order.}  2nd Ed. Springer-Verlag: Berlin.






\bibitem{HLP}
G. H. Hardy, J. E.  Littlewood, and G. P\'{o}lya,  Inequalities, 2nd edition, Cambridge University Press, 1952.



\bibitem{Horgan}
C. O. Horgan,  Plane entry flows and energy estimates for the Navier-Stokes equations, {\it Arch. Rational Mech. Anal.}, {\bf  68} (1978), 359--381.

\bibitem{HW}
C. O. Horgan and L. T.  Wheeler, Spatial decay estimates for the Navier-Stokes equations with application to the problem of entry flow, {\it SIAM J. Appl. Math.}, {\bf 35} (1978), 97--116.



\bibitem{KP}
L. V. Kapitanski and K. I. Piletskas, Spaces of solenoidal vector fields and boundary value problems for the Navier-Stokes equations in domains with noncompact boundaries. (Russian) Boundary value problems of mathematical physics, 12. Trudy Mat. Inst. Steklov. {\bf 159} (1983), 5--36.



\bibitem{LS}
O. A. Ladyzhenskaja and  V. A. Solonnikov,  Determination of solutions of boundary value problems for stationary Stokes and Navier-Stokes equations having an unbounded Dirichlet integral, {\it Zap. Nauchn. Sem. Leningrad. Otdel. Mat. Inst. Steklov. (LOMI)}, {\bf 96} (1980), 117--160.

\bibitem{Leray}
J.  Leray, Etude de diverses equations integrales non lineaires et de quelques problemes que pose l'hydrodynamique,  Journal Math. Pures Appl., {\bf 12} (1933), 1--82.



\bibitem{Liu-Wang}
J. G. Liu and W. C. Wang, Characterization and regularity for axisymmetric solenoidal vector fields with application to Navier-Stokes equation, {\it SIAM J. Math. Anal.}, {\bf 41} (2009), no. 5, 1825--1850.




\bibitem{Morimoto}
H. Morimoto,  Stationary Navier-Stokes flow in 2-D channels involving the general outflow condition. Handbook of differential equations: stationary partial differential equations. Vol. IV, 299--353, Handb. Differ. Equ., Elsevier/North-Holland, Amsterdam, 2007.

\bibitem{MF}
H. Morimoto and H. Fujita,  On stationary Navier-Stokes flows in 2D semi-infinite channel involving the general outflow condition. Navier-Stokes equations and related nonlinear problems (Ferrara, 1999). {\it Ann. Univ. Ferrara Sez. VII (N.S.)}, {\bf 46} (2000), 285--290.

\bibitem{NP1}
S. A. Nazarov and K. I. Piletskas, Behavior of solutions of Stokes and Navier-Stokes systems in domains with periodically changing cross-section. (Russian) Boundary value problems of mathematical physics, 12. Trudy Mat. Inst. Steklov., {\bf 159} (1983), 95--102.

\bibitem{NP2}
S. A. Nazarov and K. I. Piletskas, The Reynolds flow of a fluid in a thin three-dimensional channel, Litovsk. Mat. Sb., {\bf 30} (1990), no. 4, 772--783.

\bibitem{NP3}
S. A. Nazarov and K. I. Piletskas,  On the solvability of the Stokes and Navier-Stokes problems in the domains that are layer-like at infinity, {\it  J. Math. Fluid Mech.}, {\bf 1} (1999),  78--116.



\bibitem{Pileckas}
K. I.  Pileckas,  On the asymptotic behavior of solutions of a stationary system of Navier-Stokes equations in a domain of layer type. {\it Mat. Sb.}, {\bf 193} (2002), 69--104.




\bibitem{Rabier1}
P. J. Rabier,  Invertibility of the Poiseuille linearization for stationary two-dimensional channel flows: symmetric case, {\it J. Math. Fluid Mech.}, {\bf 4} (2002),  327--350.


\bibitem{Rabier2}
P. J. Rabier,  Invertibility of the Poiseuille linearization for stationary two-dimensional channel flows: nonsymmetric case, {\it J. Math. Fluid Mech.}, {\bf 4} (2002), 351--373.







\bibitem{WXX1}Y. Wang and C. Xie, Uniform Structural stability of Hagen-Poiseuille flows in a pipe, arXiv:1911.00749, 2019.

\bibitem{WX1}
Y. Wang and C. Xie,  Structural stability of Poiseuille flows in a pipe for the Navier-Stokes equations with Navier-slip boundary conditions, preprint, 2018.





\end{thebibliography}
\end{document}